\documentclass[12pt]{article}
\usepackage[
  left=1in,
  right=1in,
  top=1in,
  bottom=1in
]{geometry}

\usepackage{graphicx} % For including images with \includegraphics
\graphicspath{{./images/}} % Set default image path
\usepackage{url} % For formatting URLs (useful if not using hyperref)
\usepackage{mathtools}
\usepackage[dvipsnames]{xcolor} % Extended color names
\usepackage[linesnumbered,ruled,algosection,vlined]{algorithm2e}
\SetKwInOut{postcondition}{Postcondition}
\SetKwInOut{precondition}{Precondition}
\SetKwInOut{correctness}{Successful if}
\SetArgSty{textrm}

\SetCommentSty{mycommfont}
\let\oldnl\nl% Store \nl in \oldnl
\newcommand{\nonl}{\renewcommand{\nl}{\let\nl\oldnl}}% Remove line number for one line

\definecolor{darkblue}{rgb}{0.0, 0.0, 0.55}
\definecolor{cobalt}{rgb}{0.0, 0.28, 0.67}

\newcommand{\demph}[1]{\emph{{\color{RoyalBlue}#1}}}

\usepackage{tikz}
\usetikzlibrary{3d}
\usepackage{tikz-3dplot}
\usepackage{pgfplots}
\usetikzlibrary{arrows.meta,positioning}

\usepackage{amsmath, amsfonts, amsthm, amssymb} % AMS math packages
\usepackage{mathtools} % Enhancements to amsmath
\usepackage[colorlinks=true, allcolors=blue]{hyperref}
 
\usepackage[nameinlink]{cleveref} % For smart cross-referencing (e.g., \cref{eq:1})
\theoremstyle{definition}
\newtheorem{theorem}{Theorem}

\newtheorem{remark}[theorem]{Remark}

\newtheorem{notation}[theorem]{Notation}

\newtheorem{corollary}[theorem]{Corollary}

\newtheorem{example}[theorem]{Example}
\usepackage{enumitem} % Customizes list environments
\makeatletter
\def\saveenum{\xdef\@savedenum{\the\c@enumi\relax}}
\def\resetenum{\global\c@enumi\@savedenum}
\makeatother

\usepackage{soul} % For highlighting text (e.g., \hl{})
\usepackage{framed} % For framed environments (e.g., leftbar)
\usepackage{subcaption} % For subfigures
\usepackage{booktabs} % For professional-quality tables
\usepackage{arydshln} % For dashed lines in arrays and tables

\newcounter{todo}

\makeatletter

\newcommand\listtodoname{List of todos}
\newcommand\listoftodos{%
  \section*{\listtodoname}\@starttoc{tod}}
\makeatother
\makeatletter
\let\saved@bibitem\@bibitem
\makeatother

\makeatletter
\let\@bibitem\saved@bibitem
\makeatother

\newcommand{\codim}[1]{\operatorname{codim}(#1)}

\newcommand{\UBigGroup}{\mathcal{U}}

\newcommand{\PP}{\mathbb{P}}
\newcommand{\CC}{\mathbb{C}}

\let\oldtextit\textit 
\renewcommand\demph[1]{\oldtextit{\color{RoyalBlue}#1}}
\definecolor{RoyalBlue}{cmyk}{1, 0.50, 0, 0}

\newcommand{\dPrime}{d'}

\newcommand{\bfu}{\mathbf{u}}

\newcommand{\EE}{\mathbb{E}}
\title{Rigid homotopies for sampling from algebraic varieties: a Waring structure complexity model}

\author{Abigail R. Jones, 
 Kisun Lee,
Jose Israel Rodriguez}
\date{}
\newcommand{\threeByThree}[9]{
        \begin{bsmallmatrix}
        #1& #2&#3\\
        #4& #5&#6\\
        #7& #8&#9\\
        \end{bsmallmatrix}}

\newcommand{\lambdazeros}{B} 

\makeatletter
\def\blfootnote{\gdef\@thefnmark{}\@footnotetext}
\makeatother
\newcommand{\intMeasure}[2]{\int_{#1}{#2{\,}}}
\begin{document}

\maketitle
\begin{abstract}
Polynomial system solving has seen major progress in both theory and practice over the past decade. 
A landmark achievement was addressing Smale's 17th problem, establishing average-case polynomial-time algorithms for computing approximate solutions of polynomial systems via homotopy continuation. 
Recent improvements in complexity bounds for these algorithms led to the development of rigid homotopy methods. 
In this article, 
we prove a new complexity result for rigid homotopies for polynomial systems with Waring representations of prescribed length. In addition, we provide the first computational experiments for rigid homotopies using a preliminary implementation. 
\end{abstract}

\blfootnote{\textup{2000} \textit{Mathematics Subject Classification}:
Primary 68Q25; Secondary 65H10, 65H14.}

\section{Introduction}

Polynomial system solving has seen major progress in both theory and practice over the past decade. A central driving force of these advances is \demph{homotopy continuation}, which has emerged as an efficient tool for large-scale computational problems. On the theoretical side, homotopy continuation played a key role in addressing Smale’s 17th problem, which asks whether a root of a system of $n$ complex polynomial equations can be approximated, on average, in polynomial time by a uniform algorithm \cite[Problem 17]{smale1998mathematical}. A sequence of works \cite{beltran2008smale,beltran2011fast,burgisser2011problem,shub2009complexity}, culminating in \cite{lairez2017deterministic}, established an affirmative answer and positioned homotopy continuation as a central interest in the complexity theory of numerical algebraic geometry. In particular, this line of work studies the number of arithmetic operations required to compute a root by homotopy continuation \cite{beltran2013robust,beltran2011fast,beltran2009complexity,hauenstein2016certified,lee2025priori}.

At the same time, independently of these developments, homotopy continuation has emerged as a practical tool for solving large and structured polynomial systems, \cite{duff2023plmp,lee2023polyhedral,lindberg2025estimating,lindberg2022distribution} to name a few. These practical successes highlight a different aspect of homotopy continuation as a computational paradigm.

Despite their common foundation, these two perspectives, complexity-theoretic and computational, have evolved with different emphases. The former focuses on average-case guarantees under randomized dense models for finding a single root, while the latter is driven by structured systems, empirical performance, and often the computation of several roots at once. 
As a result, the underlying problem formulations differ: the homotopy continuation setting used in complexity analysis may not directly reflect that used in practice, and vice versa.

We investigate the practical potential of \demph{rigid homotopy} \cite{BCL2023-rigid-two,Lairez2020-rigid-one} as an effort to bridge this gap. 
In practice, polynomial systems are often structured; 
we focus in particular on systems with \demph{Waring-type representations}. 
In this setting, we develop a complexity analysis of these structured inputs, provide the first implementation of rigid homotopy for this model, and support our study with computational experiments that explore its behavior in practice.

The rigid homotopy emerged from efforts to refine the complexity analysis of homotopy continuation in the context of Smale’s 17th problem. In \cite{Lairez2020-rigid-one}, a rigid homotopy algorithm was introduced to solve random dense polynomial systems in quasi-linear average complexity in terms of the number of all monomials in the system. The sequel \cite{BCL2023-rigid-two} suggests an insight that complexity may depend on how the polynomials themselves are represented and evaluated. For systems in \demph{algebraic branching programs} (ABPs) with low evaluation complexity, \cite{BCL2023-rigid-two} refines the analysis using the \demph{black-box evaluation model}, where the cost is measured by the number of evaluation queries, denoted $e(f)$, rather than the number of monomials. For Gaussian random ABPs of degree $D$ in $n$ variables, it is shown that the associated condition measure admits a polynomial bound in $n$ and $D$ \cite[Theorem 1.5]{BCL2023-rigid-two}. Combined with the general complexity estimate, this implies that such systems can be solved in polynomial time on average \cite[Corollary 1.3]{BCL2023-rigid-two}.

We consider Waring-type representations as a natural model of structured inputs with low evaluation complexity. While they differ from ABPs, they admit efficient evaluation and arise naturally in applications, for instance in polynomial neural networks \cite{FRWY-activation-threshold,KTB2019,KLW2024}, which consist of linear transformations followed by power-type nonlinearities. 

Our main result shows that for 
random degree $D$ polynomials in $n+1$ variables with a Waring decomposition of length $r$, the expected condition number is polynomially bounded in $n$ and $D$, up to a factor depending on $r$ (\Cref{thm:waring-condition-bound}). Moreover, this factor approaches $1$ as $r \to \infty$, so that the resulting bound matches the polynomial dependence on $n$ and $D$.

\Cref{thm:waring-condition-bound} highlights that the favorable complexity behavior of rigid homotopy extends beyond dense or ABP-based models to a natural class of structured polynomials. Our analysis shows that such a structure may still achieve polynomial average complexity, with a dependence on the decomposition length $r$, and that this additional cost vanishes as $r$ grows. In this way, the result supports the practicality of rigid homotopy for a broader class of polynomial systems, with respect to representation efficiency.

This paper is organized as follows. In \Cref{s:homotopy-continution} we recall important facts of homotopy continuation. This includes the definition of the input-output correspondence for rigid homotopies and solution paths. In \Cref{s:previous-rigid-homotopy-algorithm} we provide an overview of continuation algorithms (\Cref{alg:rigid_continuation,alg:rigid_sampling,alg:gamma-estimate}) based on \cite{BCL2023-rigid-two}. The bound of the average complexity for these algorithms for a polynomial in Waring representation is established in \Cref{s:waring-section}. In~\Cref{s:comp-results}, we present a proof of concept implementation of the rigid homotopy and report computational results. These are used to motivate the future directions discussed in the conclusion, \Cref{s:conclusion-outlook}.

\section{Homotopy continuation}\label{s:homotopy-continution}

We review homotopy continuation and introduce the rigid homotopy framework. The focus is on the input-output correspondence 
underlying continuation methods and on the construction of rigid homotopy paths.
We then recall the $\gamma$-number, which is a key tool for analyzing the complexity of continuation algorithms.

\subsection{Straight-line homotopies}

We consider homotopy continuation in the projective setting, following the framework developed in \cite{lairez2017deterministic}. Let $H_d$ denote the complex vector space of homogeneous polynomials of degree $d$ in the variables $z_0,\dots,z_n$. Equip $H_d$
with the \demph{Bombieri-Weyl Hermitian inner product}, for which the monomial basis $\{z_0^{a_0}\cdots z_n^{a_n}\}$ is orthogonal and satisfies 
\[\|z_0^{a_0}\cdots z_n^{a_n}\|_W^2=\frac{a_0!\cdots a_n!}{{(a_0+a_1+\cdots+a_n)!}}.\] 
Therefore, $H_{d_1}\times \cdots \times H_{d_n}$ is the complex vector space of homogeneous polynomial systems in $n+1$ variables with degrees $d_1,\dots,d_n$, respectively. This space has the Hermitian inner product induced from the Hermitian inner product of each factor. We denote by $\mathbb S\subset H_{d_1}\times \cdots \times H_{d_n}$ the corresponding unit sphere, consisting of all systems of norm~$1$. 

The Hermitian inner product induces a Riemannian metric on $\mathbb S$, which we denote by $d_{\mathbb S}$. For $f,g\in\mathbb S$, the distance $d_{\mathbb S}(f,g)$ is given by the angle between $f$ and $g$, that is, $\cos d_{\mathbb S}(f,g)=\Re\langle f,g\rangle$, where $\Re$  denotes the real part of $\langle f,g\rangle$. Given a start system $g\in\mathbb S$ and a target system $f\in\mathbb S$, the homotopy path $h_t$ is defined as the geodesic on the unit sphere of the coefficient space connecting $g$ and $f$:
\begin{equation}\label{eq:geodesic-homotopy}
    h_t:= \frac{\sin(t\alpha)}{\sin\alpha}g+\frac{\sin((1-t)\alpha)}{\sin \alpha}f
\end{equation}
where $\alpha=d_{\mathbb S}(f,g)\in [0,\pi)$ is the spherical distance between $f$ and $g$.

The geodesic homotopy \eqref{eq:geodesic-homotopy} is projectively equivalent to the linear homotopy $tg + (1-t) f$ in the affine space $\mathbb C^n$. In practical implementations of homotopy continuation, it is common to fix a start system $g$ that is easy to solve, such as a B\'ezout start system $g=(z_1^{d_1}-1,\dots,z_n^{d_n}-1)$.

However, fixing a particular start system is not suitable for condition-based complexity analysis, which underlies Smale’s 17th problem. In this setting, the objective is to analyze the average complexity of solving a random target system, without relying on the solution of any specific start system. For this reason, randomized constructions were introduced to sample a root from a suitably distributed start system (so-called \demph{good start pair}). See, for instance, \cite[Section 2.3]{beltran2011fast} and \cite[Section 2.3]{lairez2017deterministic}.

In general, a homotopy continuation involves a parameter space of inputs and a corresponding output space. Together, these form an \demph{input-output correspondence}, also known as an incidence variety or solution space: each parameter value yields a system which determines the associated solutions. For example, in the straight-line homotopy \eqref{eq:geodesic-homotopy}, the input space is the coefficient sphere $\mathbb{S}$, and the output space is~$\PP^n$.

In straight-line continuation, the homotopy explores the full space $\mathbb{S}$, so each step may move the system in an arbitrary direction. This is natural, but it also creates opportunities for the path to pass near ill-conditioned regions, leading to small step sizes or even path failures in numerical methods. It motivates a continuation strategy that restricts the parameter space to produce paths with more controlled geometric behavior.

\subsection{Rigid homotopies}\label{ss:rigid-homotopy-background}

The \demph{rigid homotopy} replaces the input space by the product of unitary groups
\[\UBigGroup := U(n+1)^n,\]
where $U(n+1)$ denotes the group of $(n+1)\times(n+1)$ unitary matrices, and the output space remains $\PP^n$.

\newcommand{\startRoot}{\zeta}
\newcommand{\rigidCorrespondence}[1]{\mathcal{R}(#1)}
\newcommand{\gamPath}{\phi}

Fix a homogeneous polynomial system $f=(f_1,\dots,f_n)$ with $f_i \in H_{d_i}$. We regard its zero set $V(f)$ as a subset of projective space $\PP^n$. For $\bfu=(u_1,\dots,u_n)\in \UBigGroup$, define the action
\[\bfu\cdot f := (f_1\circ u_1^{-1},\dots,f_n\circ u_n^{-1})\] 
where each $u_i$ acts only on the variables of $f_i$. The \demph{rigid input-output correspondence} associated to $f$ is then
\begin{equation}\label{eq:input-output-correspondence}
\rigidCorrespondence{f}:=\left\{(\bfu,\zeta)\in \UBigGroup\times \PP^n\mid(\bfu\cdot f)(\zeta)=0\right\}.
\end{equation}
The input space $\UBigGroup$ has dimension fixed no matter the degree of $f$, $\dim_\mathbb{R} U(n+1)^n=n(n+1)^2$. On the other hand, the space of $n$ homogeneous polynomials of degree $D$ in $n+1$ variables has dimension $n\binom{n+D}{D}$. Hence, rigid homotopies restrict the deformation to a low-dimensional family. This preserves much of the system's geometric structure and reduces the likelihood of encountering ill-conditioned regions along the path.

Similar to the straight-line homotopy, a rigid homotopy path is specified by a continuous geodesic in the input space. Since the input space of $\rigidCorrespondence{f}$ is an algebraic group with a Lie algebra, a path in the rigid homotopy can be described by skew-Hermitian matrices. We define the space of $(n+1)\times (n+1)$ skew-Hermitian matrices
\[
\mathfrak{u}(n+1)
:=
\{X \in M_{n+1}(\mathbb{C}) \mid X^H = -X \}
\]
where $X^H$ denotes the conjugate transpose of $X$. Then, the Lie algebra of $\UBigGroup$ is $\mathfrak{u}(n+1)^n$. For $A=(a_1,\dots,a_n)\in\mathfrak{u}(n+1)^n$ with $\frac{1}{2}\sum_i\|a_i\|_{\mathrm{Frob}}^2\leq 1$, we define the \demph{rigid homotopy parameter path} from $\exp(A)\in\UBigGroup$ to the identity in $\UBigGroup$ by
\begin{equation}\label{eq:path-parameter-space}
\gamPath:[0,1]\to \UBigGroup,\quad t\mapsto 
\exp(tA)=(\exp(ta_1),\dots,\exp(ta_n)).\footnote{Our convention is to take $t=1$ to $t=0$ so $\gamPath(0)$ is the endpoint of the path unlike that of \cite{BCL2023-rigid-two,Lairez2020-rigid-one}}
\end{equation}
Note that we impose $\frac{1}{2}\sum_i\|a_i\|_{\mathrm{Frob}}^2\leq 1$ to ensure that the path has bounded speed. The path $\gamPath(t)$ in the input space induces a corresponding family of polynomial maps $h_t := \gamPath(t)\cdot f$. Since $h_0 = f$, tracking from $t=1$ to $t=0$, the solution analytically continues to $f=0$ along this family. We write $\zeta_t$ for the solution to $h_t=0$ obtained by analytic continuation. This produces a lifted path in the rigid input-output correspondence:
\begin{equation}\label{eq:lifted-path}
t \mapsto \left(\gamPath(t),\zeta_t\right)\in \rigidCorrespondence{f}.
\end{equation}
By construction $h_t(\zeta_t)=0$ for all $t\in[0,1]$, and we call this path a \demph{rigid homotopy solution path} with the \demph{target solution} $\zeta_0$ of the solution path.

\begin{example}\label{ex:rigid-solution-path}
For example, consider the polynomial system $f=(f_1,f_2)\in H_2\times H_4$ where

\newcommand{\plugx}{z_2}
\newcommand{\plugy}{z_1}
%% deliberately set plugx to z_2

\newcommand{\plugz}{z_0}
\newcommand{\rone}{.9}
\newcommand{\cyone}{1.88}
\newcommand{\cxone}{3.2}
\newcommand{\xzero}{(-.4)}

$f_1=\left(\frac{\plugx-\xzero \plugz}{\cxone}\right)^{2}+\left(\frac{\plugy}{\cyone}\right)^{2}-(\rone \plugz)^{2}$ and

\newcommand{\rtwo}{.85}
\newcommand{\cy}{8.7}
\newcommand{\cshift}{.1}
\newcommand{\cx}{4.25}

$f_2=
\left(\plugx^{2}+\left(\frac{\plugy}{\cy}\right)^{2}-(\rtwo \plugz)^{2}\right)
\left(\left(\frac{\plugx}{\cx}\right)^{2}+\plugy^{2}-(\rtwo \plugz)^{2}\right)+(\cshift\plugz)^4.$

The intersection of $V(f_1)$ and $V(f_2)$ are the eight points of intersection of
the blue ellipse and
the silver quartic curve as seen in \Cref{fig:rigid-homotopy-solution-path}. 
The purple surface is contained in the rigid homotopy correspondence where the horizontal axis is the $t$-axis.
Let $A_1$ be the zero matrix and $A_2=\threeByThree{0}{0}{0}{0}{0}{\theta}{0}{-\theta}{0}$. 
Each orange curve denotes a solution path in $\rigidCorrespondence{f_1,f_2}$ and has a target solution in $V(f_1,f_2)$.
\end{example}

The construction of such a path depends on a good start pair of the rigid homotopy. In practice, the good start pair is not fixed deterministically but is obtained through randomized constructions (see \Cref{ss:sampling}).

\begin{figure}[hbt!]
    \centering
    \includegraphics[width=0.6\linewidth]{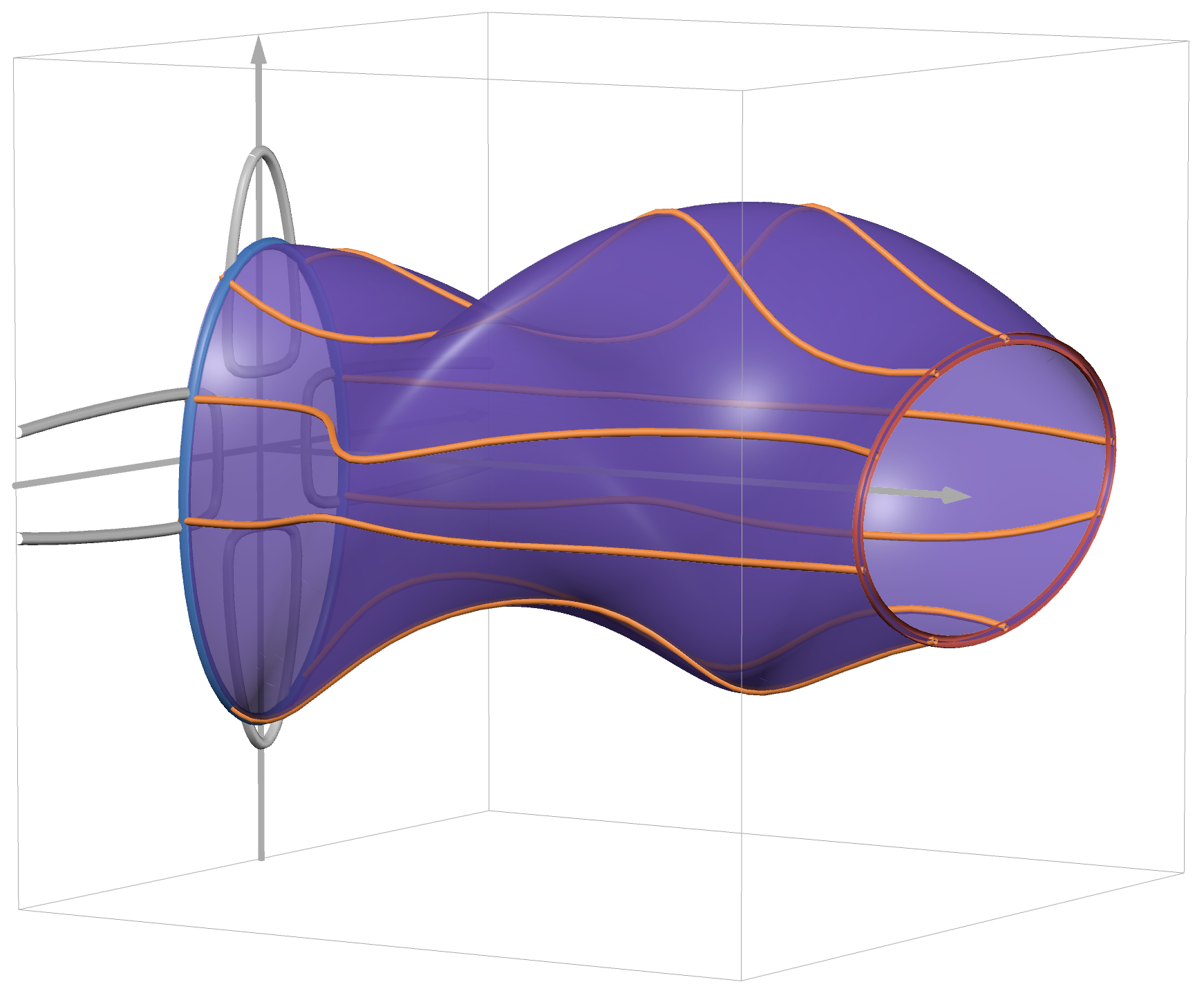}
    %\url{https://www.desmos.com/3d/wbyxon89hg}
    %\includegraphics[width=0.5\linewidth]{-code_images/rigid-homotopy-solution-path.png}
    \caption{Eight rigid homotopy solution paths where $V(f_1)$ is the ellipse and $V(f_2)$ is a quartic curve on the affine chart where $z_0=1$.}
    \label{fig:rigid-homotopy-solution-path}
    %\url{https://www.desmos.com/3d/twmvovh4r1}
\end{figure}

\subsection{Condition numbers and $\gamma$-numbers}\label{ss:gamma-numbers}

\newcommand{\gammaFrob}{\gamma_{\operatorname{Frob}}}
\newcommand{\gammaHatFrob}{\hat{\gamma}_{\text{Frob}}}

To track a solution path $\zeta_t$ of the homotopy $h_t$ in general, we iteratively advance along the path and apply Newton's method to recover the nearby exact solution. We recall the quantities that govern the convergence of Newton's method.

For a homogeneous polynomial $f\in H_d$ and a point $\zeta\in \PP^n$, we define \demph{Smale's $\gamma$-number} \cite[Chapter 8]{blum2012complexity} by
\begin{equation}\label{eq:gamma}
\gamma(f,\zeta):=\sup_{k\geq 2}\left\|\frac{1}{k!}d_zf(\zeta)^{\dagger}d_z^kf(\zeta)\right\|^{\frac{1}{k-1}}   
\end{equation} 
where $d_z f(\zeta)^{\dagger}$ denotes the Moore-Penrose pseudo-inverse of $d_z f(\zeta)$ that acts as the inverse of $d_zf(\zeta)$ on $\zeta^\perp$, and $d_z^k f(\zeta)$ is a symmetric tensor whose entries are the $k$-th partial derivatives of $f$ evaluated at $\zeta$. The $\gamma$-number measures the local nonlinearity of $f$ near $\zeta$. More precisely, it quantifies the size of the higher derivatives relative to the inverse Jacobian and therefore controls the radius within which Newton’s method converges quadratically. A small $\gamma(f,\zeta)$ indicates that the system behaves almost linearly near $\zeta$, while a large $\gamma(f,\zeta)$ reflects stronger curvature and a smaller Newton convergence region.

The same definition \eqref{eq:gamma} extends to polynomial maps $f=(f_1,\dots,f_n)\in H_{d_1}\times\cdots\times H_{d_n}$ by interpreting $d_z f(\zeta)$ as the Jacobian matrix. For the Newton correction to converge during the process of tracking a solution path $\zeta_t$ of $h_t$, the current iterate must lie within the local convergence region of the next solution point. This convergence radius is controlled by the $\gamma$-number, and is of order $\frac{1}{\gamma(h_t,\zeta_t)}$. Consequently, the step size $\Delta t$ must be chosen inversely proportional to $\gamma(h_t,\zeta_t)$: regions with larger curvature require smaller steps to maintain convergence.

In the rigid homotopy setting \eqref{eq:lifted-path}, the geometry of the input-output correspondence requires a different curvature measure, known as the split $\gamma$-number (defined below). Because rigid homotopies arise from products of unitary actions on the system's individual components, the higher-derivative contributions can be analyzed componentwise. For a polynomial map $f$ and a point $\zeta$, we consider the quantity
\[\left(\gamma(f_1,\zeta)^2+\cdots+\gamma(f_n,\zeta)^2\right)^{\frac{1}{2}}.
\]
Since the curvature is treated componentwise, the sensitivity to perturbations must be measured accordingly. For each component, this sensitivity is captured by the operator $d_z f_i(\zeta)^\dagger \circ d_z f_i(\zeta)$ projecting a vector onto the normal direction of $f_i$ at $\zeta$. To obtain a scalar quantity that measures the overall sensitivity, we combine these operators into a block-structured operator and take its norm:
\begin{equation}\label{eq:condition-number}
\kappa(f,\zeta) :=
\left\|
\begin{bmatrix}
   d_z f_1(\zeta)^\dagger \circ d_z f_1(\zeta)\\
   \vdots\\
   d_z f_n(\zeta)^\dagger \circ d_z f_n(\zeta)
\end{bmatrix}^\dagger
\right\|.
\end{equation}
This condition number becomes large when the normal directions of the hypersurfaces are nearly aligned, reflecting a loss of transversality. The \demph{split $\gamma$-number} is then defined by
\begin{equation}\label{eq:split-gamma}
    \hat{\gamma}(f,\zeta)
:=
\kappa(f,\zeta)
\left(\gamma(f_1,\zeta)^2+\cdots+\gamma(f_n,\zeta)^2\right)^{\frac{1}{2}}.
\end{equation}
It captures both the componentwise curvature and the sensitivity of the solution, thereby governing step size control in rigid homotopy continuation.

\medskip
The operator norm $\left\|\frac{1}{k!}d_zf(\zeta)^{\dagger}d_z^kf(\zeta)\right\|$ in \eqref{eq:gamma} is notoriously difficult to compute. The main difficulty comes from $d_z^kf(\zeta)$ which is a multilinear map. However, for all $k$, we can bound the operator norm with tractable quantities:
\[
\left\|\frac{1}{k!}d_zf(\zeta)^{\dagger}d_z^kf(\zeta)\right\|\leq
\frac{1}{k!}\|d_zf(\zeta)\|^{-1}\left\|d_z^kf(\zeta)\right\|_{\mathrm{Frob}}\]
Hence, for a polynomial $f\in H_d$ and $\zeta\in\PP^n$, we define
\begin{equation}
    \label{eq:Frob-gamma}%\gamma_{\mathrm{Frob}}
    \gammaFrob
    (f,\zeta):=\sup_{k\geq 2}\left(\frac{1}{k!}\|d_zf(\zeta)\|^{-1}\left\|d_z^kf(\zeta)\right\|_{\mathrm{Frob}}\right)^{\frac{1}{k-1}}.
\end{equation}
We compute $\gammaFrob(f,\zeta)$ by using the property: for a homogeneous polynomial $p$ of degree $\ell$,
\begin{equation}\label{eq:Frob-property}
    \|p(\zeta+z)\|_W=\left\|\frac{1}{\ell!} d_{z}^\ell p(\zeta)\right\|_{\mathrm{Frob}}
\end{equation} 
where $\|p(\zeta+z)\|_W$ is the Bombieri-Weyl norm of the homogeneous degree $\ell$ part of the shifted polynomial $p(\zeta+z)$ (see \cite[Lemma 30]{Lairez2020-rigid-one}).

To summarize, each $\gamma(f_i,\zeta)$ in
the split $\gamma$-number \eqref{eq:split-gamma} is upper bounded
via the Bombieri-Weyl norm of a shifted polynomial $z\mapsto \zeta+z$. Our proofs in \Cref{s:waring-section} involve these ideas.

\section{Rigid homotopy algorithm}\label{s:previous-rigid-homotopy-algorithm}

In this section, we provide an overview of the rigid homotopy algorithm in the black-box evaluation model. The primary focus is to outline the \textsc{BoundedBlackBoxNC} routine introduced in \cite[Algorithm 3]{BCL2023-rigid-two}. It proceeds in three main phases, which we describe in turn.

\subsection{Sampling}\label{ss:sampling}
In this step, we sample a 
random unitary $\mathbf{u}\in \UBigGroup$  
and a start solution $\zeta\in \PP^n$ of the system $\mathbf{u}\cdot f$. This step samples a good start pair and involves the geometry of rigid solution varieties. This sampling provides the rigid homotopy parameter path \eqref{eq:lifted-path} where $(\phi(1),\zeta_1)=(\mathbf{u},\zeta)$.

To sample a pair $(\mathbf{u}, \zeta)$ from the rigid input-output correspondence $\rigidCorrespondence{f}$, we adopt the sampling strategy introduced in \cite[Algorithm 1]{Lairez2020-rigid-one}. The procedure consists of three steps:

\begin{enumerate}[nosep]
    \item We generate a random start solution $\zeta$ from the intersection of random linear subspaces $h_1, \dots,h_n$ in $\PP^n$. 
    \item Independently, we sample random points $y_i$ from each variety $V(f_i)$. 
    \item Finally, we determine the unitary transformations $\mathbf{u} = (u_1, \dots, u_n)$ that map each source configuration $(y_i, \mathbb{T}_{y_i}V(f_i))$ to the target configuration $(\zeta, h_i)$.
\end{enumerate}
This construction guarantees that $(\mathbf{u}, \zeta)$ belongs to $\rigidCorrespondence{f}$ and follows the standard distribution required for the complexity analysis. The detailed procedure is outlined in \Cref{alg:rigid_sampling}.

\begin{algorithm}[hbt!]
\DontPrintSemicolon

\caption{Sampling from the rigid input-output correspondence \cite[Alg. 1]{Lairez2020-rigid-one}}
\label{alg:rigid_sampling}

\KwIn{A polynomial map $f=(f_1,\dots,f_n)\in H_{d_1}\times \cdots \times H_{d_n}$ with $\sum_{i=1}^n \codim{V(f_i)}=n$}
\KwOut{A pair $(\mathbf{u},\zeta)\in \rigidCorrespondence{f}$ sampled from the standard distribution}

Sample linear subspaces $h_1,\dots,h_n \subset \PP^n$ uniformly at random\;

Set $\zeta = \bigcap_{i=1}^n h_i$\;

\For{$i=1,\dots,n$}{
    Sample $y_i \in V(f_i)$ uniformly at random\;
    Set $T_i = \mathbb{T}_{y_i}V(f_i)$\;
    Sample $u_i \in \{u\in U(n+1)\mid uy_i=\zeta,\; uT_i=h_i\}$ uniformly at random\;
}

Set $\mathbf{u} = (u_1,\dots,u_n)$\;
\Return{$(\mathbf{u},\zeta)$}
\end{algorithm}
\newcommand{\GammaEstimate}{\Gamma_{\operatorname{estimate}}}
\subsection{$\gammaFrob$-number estimation}
For a homogeneous polynomial $f \in H_d$, the identity \eqref{eq:Frob-property} enables the Frobenius norm of $f$ to be evaluated in linear time in the dense setting of \cite{Lairez2020-rigid-one}. In contrast, for the black-box evaluation model, because the shift $z \mapsto \zeta + z$ expands each monomial of $f$, the number of monomials in the shifted polynomial grows combinatorially with the degree. Since the number of degree $D$ monomials in $n+1$ variables is not polynomially bounded in $n$ and $D$, the evaluation of $\|f(\zeta+z)\|_W$ may not admit a polynomial-time bound. To address this, \cite{BCL2023-rigid-two} introduced the probabilistic estimate of $\gamma_{\mathrm{Frob}}$ to obtain polynomial-time complexity in the black-box evaluation model.

That is, \cite[Algorithm 2]{BCL2023-rigid-two} provides an algorithm to estimate the Frobenius norm with a nonzero probability of failure in at most polynomial-time complexity. See \Cref{alg:gamma-estimate} and \cite[Theorem 3.3]{BCL2023-rigid-two}.

\begin{algorithm}[hbt!]
\DontPrintSemicolon

    \caption{Probabilistic estimation of $\gamma_{\mathrm{Frob}}$  \cite[Alg. 2]{BCL2023-rigid-two}}
    \label{alg:gamma-estimate}
    
    \KwIn{A homogeneous polynomial $f\in H_D$ given as a black-box evaluation program, a point $\zeta\in \mathbb{P}^{n}$, and a constant $\epsilon\in (0,1)$}
    \KwOut{An estimation $\GammaEstimate\in \mathbb{R}$ such that $\gamma_{\mathrm{Frob}}(f,\zeta)\leq \GammaEstimate\leq 192n^2D\gamma_{\mathrm{Frob}}(f,\zeta)$ with probability at least $1-\epsilon$}
    
    Set $h(z) = f(\zeta+z)$ by black-box evaluation program\;
    Set $s = \left\lceil1+\log_2\frac{D}{\epsilon}\right\rceil$\;
        \For{$i =1,\dots,s$}{
        Set $w_i$ a random uniformly distributed element in a unit ball in $\mathbb{C}^{n+1}$\;
        Compute $h_2(w_i),\dots,h_D(w_i)$, where $h_k$ is the degree $k$ homogeneous part of $h$\;
    }
    \Return{$\max\limits_{2\leq k \leq D}\left\{\left(
    \displaystyle
    \frac{(32nk)^k}{s\|d_{z}h(\boldsymbol{0})\|^2}{n+k+1\choose k}\sum\limits_{i=1}^s|h_k(w_i)|^2\right)^{\frac{1}{2k-2}}\right\}$}
\end{algorithm}

\begin{example}[Estimating $\gamma_{\mathrm{Frob}}$ for a conic]\label{ex:estimateing-for-conic}
    We provide an example to illustrate the core concepts of \Cref{alg:gamma-estimate}.     Consider the homogeneous quadratic polynomial
\[
f = x_0^2 - x_1^2 - x_2^2 \in H_2,
\]
with $n=2$, and choose the point $\zeta=[1:1:0]\in V(f)$. Expanding $f$ around $\zeta$, we obtain
\[
h(z) := f(\zeta+z)
%= (1+z_0)^2-(1+z_1)^2-z_2^2
= 2(z_0-z_1)+z_0^2-z_1^2-z_2^2,
\]
and we find the homogeneous components of $h$ are
\[
h_1(z)=2(z_0-z_1), \qquad
h_2(z)=z_0^2-z_1^2-z_2^2.
\]
Therefore, from $h_1(z)$, we find $\|d_zh(\boldsymbol{0})\|^2=8$.

Let $\epsilon=\frac{1}{4}$. Then, \Cref{alg:gamma-estimate} prescribes $s=\left\lceil 1+\log_2\frac{2}{\epsilon}\right\rceil=4$. We randomly sample \begin{align*}
w_1&=(0.4275-0.6339i,-0.3085+0.1682i,-0.4613-0.0327i),\\
w_2&=(-0.1376+0.3953i,-0.6960-0.4021i,-0.2992-0.0088i),\\
w_3&=(-0.1258-0.1158i,0.5414-0.3448i,0.1793-0.4512i),\\
w_4&=(-0.0942-0.7393i,-0.1586-0.3602i,0.2997+0.2197i).
\end{align*} 
in the unit ball of $\mathbb{C}^3$. This sampling is used to approximate the Bombieri-Weyl norm of the homogeneous component $h_2$ via Monte-Carlo estimation.
The estimate is 
\[
\GammaEstimate
=
\left(
\frac{(32\cdot 2\cdot 2)^2}{4\cdot 8}
\binom{5}{2}
\sum_{i=1}^4 |h_2(w_i)|^2
\right)^{\frac{1}{2}}.
\]
This gives $\GammaEstimate\approx 95.4$. With this choice of $\epsilon$, depending on the sampled $w_i$'s, $\Gamma$ satisfies the bound 
\begin{equation}\label{eq:theoretical-guarantee-example}
\gammaFrob
(f,\zeta)\leq \GammaEstimate\leq 192n^2D
\gammaFrob
(f,\zeta).
\end{equation}
with probability at least $1-\epsilon=\frac{3}{4}$. 
Equation \eqref{eq:theoretical-guarantee-example} gives bounds on $\gammaFrob(f,\zeta)$:
\begin{equation}\label{eq:theoretical-guarantee-bounds}
0.0621 \leq
\gammaFrob
(f,\zeta)
\leq  95.4\,.
\end{equation}

We compare the previous bounds with the exact value of $\gammaFrob(f,\zeta)$. Since $f$ is quadratic, only the term $k=2$ contributes in the definition of $\gammaFrob(f,\zeta)$ in~\eqref{eq:Frob-gamma}. 
We first consider $\|d_zf(\zeta)\|^{-1}$. 
Let $\hat{\zeta}=\frac{1}{\sqrt{2}}(1,1,0)$ be a unit representative of $\zeta$. 
Then, we have $d_{z}f(\hat\zeta)=(\sqrt{2},-\sqrt{2},0)$, and for a tangent vector $v=(v_0,v_1,v_2)\in \hat\zeta^\perp$ satisfying $v_0+v_1=0$, one obtains $\left(d_{z}f(\hat\zeta)\right)\cdot v=\sqrt{2}v_0-\sqrt{2}v_1=2\sqrt{2}\,v_0$. The maximum of $2\sqrt{2}v_0$ happens when $v_2=0$, and so we have $\|v\|=2|v_0|^2=1$. Hence, the operator norm is $2$ since $|v_0|=\frac{1}{\sqrt{2}}$, and it yields $\|d_zf(\zeta)|_{\hat{\zeta}^\perp}\|^{-1}=\frac{1}{2}$. On the other hand, the second derivative satisfies $\frac{1}{2}d_{z}^2 f(\hat\zeta)(v,v)=v_0^2 - v_1^2 - v_2^2$. Restricting to $\hat\zeta^\perp$ (which satisfies $v_0+v_1=0$), it becomes $\frac{1}{2}d_{z}^2 f(\hat\zeta)(v,v)=-v_2^2$. Viewing this as a quadratic form, the only nonzero coefficient is that of $z_2^2$, which has magnitude $1$; hence, $\left\|\frac{1}{2!}d^2_zf(\zeta)\right\|_{\mathrm{Frob}}=1$. Therefore, $\gammaFrob(f,\zeta)
=
\frac{1}{2}$. The estimate $\GammaEstimate$ obtained above is consistent with the theoretical guarantee \eqref{eq:theoretical-guarantee-example}.
\end{example}

\subsection{Continuation}

The steps above are assembled into the continuation algorithm to find a root of a homogeneous polynomial system $f=(f_1,\dots, f_n)\in H_{d_1}\times \cdots \times H_{d_n}$. Starting from the sampled pair $(\mathbf{u},\zeta)$ obtained via \Cref{alg:rigid_sampling}, we construct a path $\phi(t)=\exp(tA)$ where $\phi(1)=\exp(A)=\mathbf{u}$ and $\phi(0)$ is the identity of $\UBigGroup$. This induces a homotopy $h_t := \phi(t)\cdot f$, along which the solution is tracked from $t=1$ to $t=0$ using numerical continuation. The step size is determined by a probabilistic estimate of the $\gamma$-number \eqref{eq:gamma} obtained from \Cref{alg:gamma-estimate}. The resulting continuation procedure is described in \Cref{alg:rigid_continuation}.

\begin{algorithm}[hbt!]
\DontPrintSemicolon

    \caption{Rigid homotopy continuation  \cite[Alg. 3]{BCL2023-rigid-two}}
    \label{alg:rigid_continuation}
    
    \KwIn{A homogeneous polynomial system $f=(f_1,\dots, f_n)\in H_{d_1}\times \cdots \times H_{d_n}$, a start system $\mathbf{u} \cdot f$, a start solution $\zeta \in V(\mathbf{u}\cdot f)$ sampled using \Cref{alg:rigid_sampling}, an upper bound $K$ on the number of continuation steps, and $\epsilon\in(0,\frac{1}{2}]$.}
    \KwOut{A point $\zeta^*$ which is an approximate zero of $f$ with probability at least $1-\epsilon$ whenever the algorithm terminates after at most $K$ continuation steps.}   
    Set $t = 1$\;
    Set $z = \zeta$\;
    Set $\phi(t)$ a path with $\phi(1)=\mathbf{u}$ and $\phi(0)$ is the identity of $\UBigGroup$\;
    \While{$t > 0$}{
        Set $h_t = \phi(t)\cdot f$\;
        Set $g_i$ as an estimate returned by \Cref{alg:gamma-estimate} for each polynomial in $h_t$ with failure probability $\frac{\epsilon}{nK}$\;
        Set $\displaystyle\hat{\gamma}_{\operatorname{Frob}}(h_t, z) = \kappa(h_t, z) \sqrt{\sum_{i=1}^n g_i^2 }$\;
        Set $\displaystyle\Delta t = \frac{1}{240  \kappa(h_t, z)\hat{\gamma}_{\text{Frob}}(h_t, z)}$\; 
        Set $t = \max\{t -\Delta t, 0\}$\;
        Set $z$ to the result of Newton iterations for $h_t$ starting from $z$\;    
        }
    
    \Return{$\zeta^*=z$}
\end{algorithm}

\newcommand{\GammaFrob}{\Gamma_{\operatorname{Avg}}}

The role of $K$ is to distribute the failure probability among the finitely many runs of \Cref{alg:gamma-estimate}. If at most $K$ continuation steps are performed, then at most $nK$ probabilistic estimates are made. By assigning failure probability $\frac{\epsilon}{nK}$ to each call, the union bound implies that all estimates are simultaneously correct with probability at least $1-\epsilon$.

The complexity of the continuation procedure is governed by the local conditioning along the solution path. Since the algorithm starts from a randomly sampled root, we average the $\gamma$-numbers over the solution set of each polynomial $f_i$. This leads to the definition of the \demph{averaged $\gamma$-number} of $f_i$:
\begin{equation} \label{eq:Gamma_def}
    \GammaFrob(f_i) := \EE_{\zeta \in V(f_i)} \left[\gammaFrob(f_i, \zeta)^2\right] ^{\frac{1}{2}},
\end{equation}
where the expectation is taken over a uniformly distributed zero $\zeta$ of $f_i$. 

The definition of the averaged $\gamma$-number admits a natural extension to homogeneous polynomial systems
$f=(f_1,\dots,f_n)\in H_{d_1}\times\cdots\times H_{d_n}$.
For such a system, the averaged $\gamma$-number is given by
\begin{equation}\label{eq:averaged-gamma-system}
\GammaFrob(f):=\left(\GammaFrob(f_1)^2+\cdots +\GammaFrob(f_n)^2\right)^{\frac{1}{2}}.    
\end{equation}

\subsection{Efficient evaluation models: algebraic branching programs \& Waring~representations}
An algebraic branching program (ABP) models a polynomial $f$ as the sum of weights of all paths in a layered graph, or equivalently, as an iterated matrix multiplication $f(z) = \text{tr}(A_1(z) \cdots A_D(z))$ where the entries of the matrices $A_i(z)$ are linear forms in variables $z:=(z_0,\dots,z_n)$.
Note that $\GammaFrob(f)$  depends on the polynomial $f\in H_D$. 
For a distribution on the entries of the matrices $A_i(z)$, 
we have the expectation $\EE_{\{A_i(z)\}}[\GammaFrob(f)^2]$. The following theorem establishes an explicit bound on the averaged condition number for this class of polynomials:

\begin{theorem}[\cite{BCL2023-rigid-two}, Theorem 1.5] \label{thm:abp_gamma}
    If $f$ is a random polynomial computed by an irreducible Gaussian random ABP of degree $D$, then 
    \begin{equation}\label{eq:expectation-ABP}
        \EE_{\{A_i(z)\}}[\GammaFrob(f)^2] \leq \frac{3}{4} D^3 (D+n) \log D.
    \end{equation}
\end{theorem}

The theorem shows that for polynomials computed by ABPs, the expectation 
in~\eqref{eq:expectation-ABP}
is polynomially bounded in $n$ and $D$.
Hence, it confirms that the rigid homotopy continuation applied to random systems consisting of Gaussian random ABPs of degree at most $D$ terminates after polynomial many arithmetic operations \cite[Corollary 1.6]{BCL2023-rigid-two}. 

\medskip
On the other hand, we focus on Waring-type representations as another natural model of low evaluation complexity. A homogeneous polynomial $f \in \mathbb{C}[z_0,\dots,z_n]$ of degree $D$ is said to admit a \demph{Waring representation} of length $r$ if it can be written in the form 
\[
f(z):=L_1(z)^D+\cdots + L_r(z)^D
\]
where $L_1,\dots, L_r$ are linear forms in $z=(z_0,\dots, z_n)$. Although the expanded form of $f$ is typically dense and contains ${n+D\choose D}$ monomials, it can be evaluated by computing $r$ linear forms and raising each to the $D$-th power. In particular, the evaluation complexity satisfies $e(f)=O(r(n+D))$, which can be significantly smaller than the dense coefficient size.
Thus, Waring-type polynomials form a canonical family of evaluation-efficient structured inputs within the black-box model of rigid continuation.

For polynomial systems with a fixed Waring representation length, we prove \Cref{thm:waring-condition-bound}, an analog of \Cref{thm:abp_gamma}. Polynomial systems admitting Waring representations form a natural class of unitary-invariant structured inputs with low evaluation complexity, making them suited to the black-box complexity framework for rigid homotopies. Even if a polynomial has a high degree, when it has a Waring representation, its geometric complexity is governed by relatively less information about its rank-one summands. It means that their average complexity may be favorable. Specifically, we prove that for random polynomials with a prescribed Waring representation length, the expected condition number is polynomially bounded asymptotically, especially when the Waring representation length $r$ is sufficiently larger than the degree $D$ of the polynomial.

Although Waring representation is a seemingly simpler structural model than ABPs, their conditioning behavior is more delicate. The stronger correlations among coefficients may introduce dominance of a few summands in the high-degree power structure that does not arise in the ABP setting. As a result, achieving a polynomial complexity bound is not automatic. 
The analysis does not follow 
a straightforward adaptation of the argument for ABPs. Instead, it requires a separate treatment as illustrated in \Cref{ss:evaluating-outer-W1,ss:consequences}.

\newcommand{\vol}{\operatorname{vol}}
\newcommand{\Frob}{\operatorname{Frob}}
\newcommand{\trace}{\operatorname{tr}}
\newcommand{\diag}{\operatorname{diag}}

\section{A condition-based complexity result for a Waring structured system}\label{s:waring-section}

\newcommand{\normSquaredKthDerivative}{\left\|\frac{1}{k!}d_z^kf_L(\zeta)\right\|^2_{\Frob}}

\newcommand{\mainExpectation}{\mathbb{E}_{L,\zeta\in V(f_L)}\left[\|d_z f_L(\zeta)\|^{-2}\normSquaredKthDerivative\right]}

We devote this section to proving \Cref{thm:waring-condition-bound}, the explicit bound on the averaged condition number for a polynomial in Waring representation. In \Cref{ss:fix-notation}, we introduce the Waring structured polynomial model and fix the notation. In \Cref{ss:formulate-main-expectation}, we begin the analysis of the $\gamma$-number \eqref{eq:Gamma_def}. Using the Waring structure, we first reduce the problem to the single-integral expression \eqref{eq:ME-single-integral}, and then, by exploiting unitary invariance, reformulate it as an iterated integral in \Cref{ss:double-integral}. The inner and outer integrals are evaluated in \Cref{ss:evaluating-inner-Wr,ss:evaluating-outer-W1}, which leads to the conclusion in \Cref{ss:consequences}.

\subsection{The Waring structured polynomial model}\label{ss:fix-notation}

Fix a positive integer $r$. For any linear map $L=(L_1,\dots,L_r):\CC^{n+1}\to \CC^r$ with each linear form $L_i$, we associate the (degree $D$) homogeneous polynomial
\[f_L(z):=
\sum_{i=1}^r L_i(z)^D\in H_D.
\]
Each linear form admits the representation
\[L_i(z) = \langle \lambda_i, z \rangle 
       = \lambda_{i,0} z_0 + \cdots + \lambda_{i,n} z_n\]
where $\lambda_i \in \mathbb{C}^{n+1}$. Thus, using the usual Hermitian norm on the dual space of $\CC^{n+1}$ we have 
$\Vert L_i\Vert^2:=\sum\limits_{j=0}^n |\lambda_{i,j}|^2$.

The parameter space of Waring representations of length $r$ is
\[
W_r := \{ L = (L_1,\dots,L_r) 
    \mid L_i \in (\mathbb{C}^{n+1})^\ast \}
    \cong \mathbb{C}^{r\times (n+1)}.
\]
So, \(\dim_{\mathbb{C}} W_r = r(n+1)\), and we have $\Vert L\Vert^2 := \sum\limits_{i=1}^r\| L_i\Vert^2$.

On this space we consider the standard Gaussian probability measure, whose density with respect to Lebesgue measure $dL$ on $\mathbb{C}^{r\times(n+1)}$ is \begin{equation}\label{eq:density-L} \dPrime L := \pi^{-(n+1)r} \exp(-\Vert L\Vert^2)\, dL. \end{equation} Here, $\dPrime L$ is simply a convenient notation for this Gaussian measure.

\begin{notation}
As is common in the rigid homotopy literature, 
we place the measure before the integrand to emphasize the domain where the integration is taken.
For instance, 
$\intMeasure{X}{dx}\left(K\right)$, as well as
$\intMeasure{Y}{dy}\left(
\intMeasure{X}{dx}
K\right)$
 for an iterated integral.
\end{notation}

\begin{remark}
A related invariant of a polynomial is the Waring rank. This is the minimal $r$ for which there exists a Waring representation of length $r$. This quantity is NP hard to compute and not the objective of our interest.
\end{remark}

\subsection{Formulation of the main expectation}\label{ss:formulate-main-expectation}

We study the averaged $\gamma$-number \eqref{eq:Gamma_def} along the zero set of a random polynomial $f_L$ induced by a Gaussian linear map $L$. Our goal is to analyze its expectation over both the randomness of $L$ and the choice of a root $\zeta \in V(f_L)$, namely 
\[
\EE_L[\GammaFrob(f_L)^2]=\EE_{L,\zeta \in V(f_L)}[\gammaFrob(f_L,\zeta)^2].
\]

To bound $\EE_{L,\zeta \in V(f_L)}[\gammaFrob(f_L,\zeta)^2]$, we analyze each term of 
\[\gammaFrob
    (f,\zeta)=\sup_{k\geq 2}\left(\frac{1}{k!}\|d_zf(\zeta)\|^{-1}\left\|d_z^kf(\zeta)\right\|_{\mathrm{Frob}}\right)^{\frac{1}{k-1}}
    \]
separately. For Gaussian random $L\in W_r$ and a root $\zeta$ uniformly distributed over $V(f_L)\subset \PP^n$, we consider 
\begin{equation}\label{eq:main-expectation}
    \mainExpectation
\end{equation}
for each $k\in\{2,\dots, D\}$. We eventually bound  $\EE_{L,\zeta \in V(f_L)}[\gammaFrob(f_L,\zeta)^2]$ using the bound of the expectation \eqref{eq:main-expectation} for each $k$ in the proof of \Cref{thm:waring-condition-bound}. The factor $\normSquaredKthDerivative$ will be analyzed in
\Cref{ss:evaluating-inner-Wr}. 
We focus here on the term
$\|d_z f_L(\zeta)\|^{-2}$
where we have previously defined the linear map
\[d_zf_L(\zeta):\mathbb{C}^{n+1}\to\mathbb{C},
\qquad v\mapsto d_zf_L(\zeta)(v)=\left.\frac{d}{dt}\right|_{t=0} f_L(\zeta + t v).
\]

The expectation \eqref{eq:main-expectation} as integral is written as  
\[
\intMeasure{W_r}{\dPrime L}
\left(
    \,
    \frac{1}{\vol(V(f_L))}
    \,
    \intMeasure{V(f_L)}{d\zeta\,}
    \|d_z f_L(\zeta)\|^{-2}
    \left\|\frac{1}{k!}d_z^k f_L(\zeta)\right\|_{\Frob}^2
\right). 
\]
Using the fact that $\vol(V(f_L))=D \vol(\mathbb{P}^{n-1})$ for a dense open set of $L\in W_r$ 
we have
\[
\frac{1}{D\vol(\mathbb{P}^{n-1})}
\intMeasure{W_r}{\dPrime L}
\left(
    \intMeasure{V(f_L)}{d\zeta\,}
    \|d_z f_L(\zeta)\|^{-2}
    \left\|\frac{1}{k!}d_z^k f_L(\zeta)\right\|_{\Frob}^2
\right).
\]
\cite[Proposition 5.6]{BCL2023-rigid-two} allows us to reverse the order of integration 
to get
\begin{equation}
\frac{1}{D\vol(\mathbb{P}^{n-1})}
\intMeasure{\PP^n}{
d\zeta}\left(\int_{W_\zeta}
\dPrime L
\|\partial_L f_L(\zeta)\|^{-2}
\normSquaredKthDerivative
\right),\label{eq:ME-double-integral}
\end{equation}
where $W_\zeta:=\{L\mid f_L(\zeta)=0\}$ 
and $\partial_L f_L(\zeta)$ denotes the directional derivative
\begin{equation}\label{eq:define-partial-L}
\partial_L f_L(\zeta)(\dot L):=
\left.\frac{d}{dt}\right|_{t=0}
f_{L+t\dot L}(\zeta).
\end{equation}

\newcommand{\magicConstant}{K}
Recall that the root $\zeta$ is a zero of the polynomial $f_L$ for a fixed $L$. Although both $L$ and $\zeta$ are random variables, $\zeta$ is sampled conditionally on $L$ from the zero set $V(f_L)$. For $U\in U(n+1)$, let $U\cdot L$ denote the linear forms $(U\cdot L)_i(z)=L_i(U^{-1}z)$. Then, the corresponding root must transform as $\zeta \mapsto U\zeta$ in order to preserve the relation $f_L(\zeta)=0$. Indeed, we have $f_L(\zeta)=f_{U\cdot L}(U\zeta)$.
Since $\dPrime L$ is unitarily invariant,
we know that the integrand
$$\intMeasure{W_\zeta}{{\dPrime}L}
    \left(
    \|\partial_L f_L(\zeta)\|^{-2}\normSquaredKthDerivative
    \right)$$
in \eqref{eq:ME-double-integral}
is independent of the choice of $\zeta$,  and therefore constant in $\zeta$. 
Using the standard fact that 
$$\intMeasure{\PP^n}{d\zeta}
\left(\magicConstant
\right)= \frac{\pi}{n}\vol(\mathbb{P}^{n-1})\magicConstant$$
holds
for any constant $\magicConstant$, 
we simplify \eqref{eq:ME-double-integral} to get
\begin{multline}\label{eq:ME-single-integral}
\mainExpectation
\\=\frac{\pi}{Dn}
\intMeasure{W_\zeta}{{\dPrime}L}
    \left(
    \|\partial_Lf_L(\zeta)\|^{-2}\normSquaredKthDerivative\right).
\end{multline}

\subsection{Formulation to 
iterated integral}\label{ss:double-integral}

To evaluate the integral \eqref{eq:ME-single-integral}, we begin with analyzing the term $\|\partial_L f_L(\zeta)\|$ where the notation $\partial_L f_L(\zeta)$
is introduced in \eqref{eq:define-partial-L}.

By unitary invariance, we may assume that $\zeta = [1:0:\cdots:0]$.
There are several benefits from this assumption. First, the parameter space $W_r$ admits a natural orthogonal decomposition
\[
W_r\cong \mathbb{C}^r\oplus \CC^{r\times n}
\]
where $\CC^r$ consists of the coefficients of $z_0$ and $\CC^{r\times n}$ consists of the coefficients of $z_1,\dots,z_n$ in the linear forms $L_1,\dots,L_r$. Put simply, if \newline
$\begin{bmatrix}
    \lambda_{1,0}&\lambda_{1,1}&\dots&\lambda_{1,n}\\
    \vdots&\vdots&\ddots&\vdots\\
\lambda_{r,0}&\lambda_{r,1}&\dots&\lambda_{r,n}\\    
\end{bmatrix}
\in W_r
$,
then we set 
$
\lambdazeros:=\begin{bmatrix}
    \lambda_{1,0}\\
\vdots\\
\lambda_{r,0}    
\end{bmatrix}\in \CC^r$
and
$C:=\begin{bmatrix}
    \lambda_{1,1}&\dots&\lambda_{1,n}\\
    \vdots&\ddots &\vdots\\
\lambda_{r,1}&\dots&\lambda_{r,n}\\    
\end{bmatrix}\in \CC^{r\times n}.$
Therefore, $L(z)=z_0\lambdazeros +C\begin{bsmallmatrix}
    z_1\\ \vdots \\ z_n
\end{bsmallmatrix}$.

Secondly, the same decomposition applies to tangent vectors. Since $T_L W_r \cong W_r$, any tangent vector $\dot{L} \in T_L W_r$ can be written as $\dot L= (\dot B, \dot C)$ where $\dot B=\begin{bmatrix}
    \dot\lambda_{1,0}&\cdots &\dot\lambda_{r,0}
\end{bmatrix}^\top$ and $\dot C\in \mathbb{C}^{r\times n}$.

Lastly, under the assumption $\zeta=[1:0:\cdots:0]$, $W_\zeta$ admits the decomposition 
\begin{equation}\label{eq:W-zeta-decomposition}
W_\zeta\cong V(\lambda_{1,0}^D+\cdots + \lambda_{r,0}^D)\oplus \CC^{r\times n}.
\end{equation}

\newcommand{\iWaring}{i}

Using the decomposition \eqref{eq:W-zeta-decomposition}, we 
simplify \eqref{eq:ME-single-integral} to obtain
 a new iterated integral:
\begin{equation}\label{eq:a-new-double-integral}
\frac{\pi}{Dn}
\intMeasure{V(\lambda_{1,0}^D+\cdots + \lambda_{r,0}^D)}{\dPrime B}
\left(
    \intMeasure{\CC^{r\times n}}{{\dPrime}C}
    \|\partial_L f_L(\zeta)\|^{-2}\normSquaredKthDerivative
\right).
\end{equation}
Then, for $\zeta=[1:0:\dots:0]$ we have  
$\partial_Lf_L(\zeta)(\dot{L})=
\left.\frac{d}{dt}\right|_{t=0}
f_{L+t\dot L}(\zeta)=
D\sum\limits_{\iWaring=1}^r\lambda_{\iWaring,0}^{D-1}\dot{\lambda}_{\iWaring,0}$, and taking the squared norm gives $\|\partial_L f_L(\zeta)\|^2=D^2\sum\limits_{\iWaring=1}^r|\lambda_{\iWaring,0}|^{2(D-1)}$. Since $\|\partial_L f_L(\zeta)\|^2$ does not depend on $C$, we find 
\eqref{eq:a-new-double-integral} equals
 \begin{equation}
     \frac{\pi}{Dn}
     \intMeasure{V(\lambda_{1,0}^D+\cdots + \lambda_{r,0}^D)}{{\dPrime}B}
     \left(
     \frac{1}
     {D^2\sum\limits_{\iWaring=1}^r|\lambda_{\iWaring,0}|^{2(D-1)}}
     \intMeasure{\CC^{r\times n}}{{\dPrime}C}\normSquaredKthDerivative
     \right).\label{eq:ME-up-to-factor-double-integral}
  \end{equation}

It remains to evaluate the 
iterated integral in \eqref{eq:ME-up-to-factor-double-integral}. We evaluate the inner integral in \Cref{ss:evaluating-inner-Wr} and the outer integral in
\Cref{ss:evaluating-outer-W1}.

\subsection{Evaluating the inner integral}\label{ss:evaluating-inner-Wr}
To evaluate the inner integral in \eqref{eq:ME-up-to-factor-double-integral}, we turn to understanding
$\normSquaredKthDerivative$. 

Recalling that $\zeta=[1:0:\cdots :0]$, define the (non-homogeneous) polynomial
\[
g(z):=f_L(\zeta+z),\]
and let $g_k$ be its $k$-th homogeneous part. Then, by the property of the Frobenius norm \eqref{eq:Frob-property}, we have 
\[
\normSquaredKthDerivative=
\|g_k\|^2_W.\]
Recall that $\|\cdot\|_W$ is the Bombieri-Weyl norm, so we have additional simplifications.

\newcommand{\DD}{D}
Note that we have
\begin{align*}
    g(z_0,\dots, z_n)
    &=\sum\limits_{m=0}^D{D\choose m}(1+z_0)^{D-m}h^{(m)}(z_1,\dots, z_n)
\end{align*}
where
\[
h^{(m)}(z_1,\dots, z_n):= 
\sum_{\iWaring=1}^r \lambda_{\iWaring,0}^{D-m} (\lambda_{\iWaring,1}z_1+\cdots +\lambda_{\iWaring,n}z_n)^m.
\]
Since $h^{(m)}$ 
is homogeneous and of degree $m$ in $z_1,\dots, z_n$, we have
\begin{equation}\label{eq:g_k_sum}
    g_k(z_0,\dots, z_n)=\sum_{m=1}^k{D-m \choose k-m}{D\choose m}z_0^{k-m} h^{(m)}(z_1,\dots, z_n).
\end{equation}
From the definition of the Bombieri-Weyl norm, note that ${k \choose m}\|z_0^{k-m}p\|^2_W=\|p\|^2_W$ for any homogeneous polynomial $p(z_1,\dots, z_n)$ of degree $m\leq k$. Furthermore, the inner product of two different monomials is equal to $0$. It means that all terms in \eqref{eq:g_k_sum} are orthogonal.
Hence,
\[\|g_k\|_W^2=\sum_{m=1}^k{D-m\choose k-m}^2{D\choose m}^2  {k\choose m}^{-1}\|h^{(m)}\|^2_W.\]
Recalling the inner integral of \eqref{eq:ME-up-to-factor-double-integral}, we have
\begin{equation}\label{eq:dprimeC-sum-norm-squared-h-m-W}
\int {\dPrime}C\,\|g_k\|_W^2=\sum_{m=1}^k{D-m \choose k-m}^2{D\choose m}^2
{k\choose m}^{-1}
\int {\dPrime}C\,\|h^{(m)}\|_W^2.
\end{equation}
By \cite[Lemma 3.1]{BCL2023-rigid-two}, we have 
\[\int {\dPrime}C \|h^{(m)}\|^2_W={m+n-1\choose m}\frac{1}{\vol(\mathbb{S}(\mathbb{C}^n))}\int {\dPrime}C|h^{(m)}(z)|^2.\]

With these new expressions, we now turn to evaluating the integral. Recall that

\[h^{(m)}(z_1,\dots, z_n)=\sum_{\iWaring=1}^r \lambda_{\iWaring,0}^{D-m}(\lambda_{\iWaring,1}z_1+\cdots +\lambda_{\iWaring,n}z_n)^m\] 
where $\lambda_{\iWaring,1},\dots, \lambda_{\iWaring,n}$ are standard Gaussian and $(z_1,\dots, z_n)\in \mathbb{S}(\mathbb{C}^n)$ is a unit vector. Hence, $Y_{\iWaring}=\lambda_{\iWaring,1}z_1+\cdots +\lambda_{\iWaring,n}z_n$ is a scalar random variable in $N_\mathbb{C}(0,1)$. Furthermore, for each $\iWaring$, coefficients $\lambda_{\iWaring,1},\dots, \lambda_{\iWaring,n}$ are independent, and so each $Y_j$ is independent. Therefore,
\begin{align*}
    \int {\dPrime}C |h^{(m)}(z)|^2 
    \;=\;
    \EE\left[\left|\sum\limits_{\iWaring=1}^r\lambda_{\iWaring,0}^{D-m}Y_\iWaring^m\right|^2\right]\;=\; \sum\limits_{\iWaring=1}^r\left|\lambda_{\iWaring,0}^{D-m}\right|^2 \EE\left[\left|Y_{\iWaring}^m\right|^2\right]\;=\;m!\sum\limits_{\iWaring=1}^r|\lambda_{\iWaring,0}|^{2(D-m)}.
\end{align*}
The last equality follows since $\EE\left[|Y_\iWaring|^{2m}\right]=m!$.
It means that 
$$\int {\dPrime}C\|h^{(m)}\|_W^2=\frac{(m+n-1)!}{(n-1)!}\sum\limits_{\iWaring=1}^r|\lambda_{\iWaring,0}|^{2(D-m)}.$$
Combining this 
with \eqref{eq:ME-up-to-factor-double-integral} and \eqref{eq:dprimeC-sum-norm-squared-h-m-W}, 
we have

\newcommand{\quotientlambdazero}{\frac{\sum_{i=1}^r|\lambda_{i,0}|^{2(D-m)}}{\sum_{i=1}^r|\lambda_{i,0}|^{2(D-1)}}}

\begin{multline}\label{eq:W1-integral}
    \int_{W_\zeta}{\dPrime}L\|\partial_Lf_L(\zeta)\|^{-2}\left\|\frac{1}{k!}d^k_z f_L(\zeta)\right\|_{\Frob}^2\\
    =\intMeasure{V(\lambda_{1,0}^D+\cdots + \lambda_{r,0}^D)}{{\dPrime}B}
    \left(
    \frac{1}{D^2\sum\limits_{\iWaring=1}^r|\lambda_{\iWaring,0}|^{2(D-1)}}\int_{\CC^{r\times n}}{\dPrime}C\normSquaredKthDerivative
    \right)\\
    =\frac{1}{D^2}{D\choose k}^2\sum\limits_{m=1}^k{k \choose m}\frac{(m+n-1)!}{(n-1)!}\int_{V(\lambda_{1,0}^D+\cdots + \lambda_{r,0}^D)}{\dPrime}\lambdazeros\quotientlambdazero.
\end{multline}
Note that the last equality follows from 
\[{D-m\choose k-m}^2{D\choose m}^2{k\choose m}^{-1}\;=\;{D\choose k}^2{k\choose m}^2{k\choose m}^{-1}\;=\;{D\choose k}^2{k\choose m}.\]

\subsection{Evaluating the outer integral}\label{ss:evaluating-outer-W1}
Recall that we let $\lambdazeros=\begin{bmatrix}
    \lambda_{1,0}&\cdots &\lambda_{r,0}
\end{bmatrix}^\top\in \CC^r$. For each positive integer $m\leq D$, we now evaluate the integral 
\begin{equation}\label{eq:Fermat-integral}
\intMeasure{V(\lambda_{1,0}^D+\cdots + \lambda_{r,0}^D)}{\dPrime \lambdazeros}
\quotientlambdazero
\end{equation}
from \eqref{eq:W1-integral}. For notational convenience,
we define the integrand
\newcommand{\puRational}{q}
\[\puRational(\lambdazeros):=\quotientlambdazero\]
which is homogeneous of degree $2(D-m)-2(D-1)=2-2m$ in  $\lambdazeros=\begin{bmatrix}
    \lambda_{1,0}&\cdots &\lambda_{r,0}
\end{bmatrix}^\top$.

We let $d\lambdazeros$ denote the induced Riemannian volume measure on $V(\lambda_{1,0}^D+\cdots + \lambda_{r,0}^D)$. The measure $d'\lambdazeros$ is the Gaussian probability measure on $V(\lambda_{1,0}^D+\cdots + \lambda_{r,0}^D)$ obtained by normalizing the density $\exp(-\|\lambdazeros\|^2)$ with respect to $d\lambdazeros$, namely
\[d'\lambdazeros:=\frac{\exp(-\|\lambdazeros\|^2) d\lambdazeros}{\int_{V(\lambda_{1,0}^D+\cdots + \lambda_{r,0}^D)} \exp(-\|\lambdazeros\|^2) d\lambdazeros}.
\]

We separate the integral~\eqref{eq:Fermat-integral} into a product of radial and spherical components using polar coordinates. Writing $\lambdazeros=\rho\theta$ with $\rho=\|\lambdazeros\|\in[0,\infty)$ and $\theta\in V(\lambda_{1,0}^D+\cdots + \lambda_{r,0}^D)\cap\mathbb{S}^{2r-1}$, we obtain the following decomposition
\begin{equation}
\int_{V(\lambda_{1,0}^D+\cdots + \lambda_{r,0}^D)} d'{\lambdazeros}q({\lambdazeros})=
\frac{
\int_0^\infty
\rho^{2r-3} \exp(-\rho^2) \rho^{2-2m}  d\rho
}{
\int_0^\infty
\rho^{2r-3} \exp(-\rho^2) d\rho
}
\cdot
\frac{
\int_{V(\lambda_{1,0}^D+\cdots + \lambda_{r,0}^D)\cap\mathbb{S}^{2r-1}} q(\theta) d\sigma(\theta)
}{
\int_{V(\lambda_{1,0}^D+\cdots + \lambda_{r,0}^D)\cap\mathbb{S}^{2r-1}} d\sigma(\theta)
}
\end{equation}
where $d\sigma(\theta)$ denotes the induced surface measure on $V(\lambda_{1,0}^D+\cdots + \lambda_{r,0}^D)\cap\mathbb{S}^{2r-1}$.

The radial part simplifies to
\[
\frac{\int_0^\infty \rho^{2r-2m-1}\exp(-\rho^2)d\rho}{\int_0^\infty \rho^{2r-3}\exp(-\rho^2)d\rho},
\]
while the spherical part can be expressed as an expectation with respect to the normalized surface measure on $V(\lambda_{1,0}^D+\cdots + \lambda_{r,0}^D)\cap\mathbb{S}^{2r-1}$. Consequently,
\begin{equation}
\intMeasure{V(\lambda_{1,0}^D+\cdots + \lambda_{r,0}^D)} {d'{\lambdazeros}}q({\lambdazeros})
=
\frac{\int_0^\infty \rho^{2r-2m-1}\exp(-\rho^2)d\rho}{\int_0^\infty \rho^{2r-3}\exp(-\rho^2)d\rho}
\cdot \EE_\theta[\puRational(\theta)].
\label{eq:radial-spherical-factors}
\end{equation}

To integrate the radial factor, we set $t=\rho^2$, then
\[\int_0^\infty \rho^{2r-2m-1}\exp({-\rho^2})d\rho
\;=\;
\frac{1}{2}\int_0^\infty t^{(r-m)-1}e^{-t}dt
\;=\;
\frac{1}{2}G(r-m)\]
where $G$ is the Gamma function. Here, we note that $r>m$ for the convergence of the integral.
Likewise, the denominator in 
\eqref{eq:radial-spherical-factors} is a value of the Gamma function:
\[\frac{1}{2}G(r-1)= 
\int_0^\infty \rho^{2r-3}\exp(-\rho^2)d\rho.\]
Therefore, the radial factor is
\begin{equation}\label{eq:radial-part} \frac{\int_0^\infty \rho^{2r-2m-1}\exp(-\rho^2)d\rho}{\int_0^\infty \rho^{2r-3}\exp(-\rho^2)d\rho}
\;=\;
\frac{G(r-m)}{G(r-1)}
\;=\;
\left\{\begin{array}{ll}
    1 & \text{ if }m=1 \\
    \prod\limits_{j=2}^m\frac{1}{r-j} & \text{ if }m\geq 2 
\end{array}\right.
\end{equation}

To bound the spherical factor in \eqref{eq:radial-spherical-factors}, we consider the following inequalities
\[\EE_\theta [\puRational(\theta)]
\;\leq\;
\sup_{\theta\in V(\lambda_{1,0}^D+\cdots + \lambda_{r,0}^D)\cap\mathbb{S}^{2r-1}}\puRational(\theta)
\;\leq\; \sup_{\theta\in\mathbb{S}^{2r-1}}\puRational(\theta).\]
We set $X_j=|\theta_j|^2$. 
Then, we want to maximize 
\[\hat{\puRational}(X):=\frac{\sum_{\iWaring=1}^rX_\iWaring^{D-m}}{\sum_{\iWaring=1}^rX_\iWaring^{D-1}}\]
over the standard simplex 
\[\Delta^{r-1}:=\left\{X\in \mathbb{R}^r\mid \sum_{\iWaring=1}^r X_\iWaring=1, X_\iWaring\geq 0\right\}\]
in $\mathbb{R}^r$.
Note that the critical points happen when there are exactly $k$ nonzero coordinates which are all equal to $\frac{1}{k}$ for $1\leq k\leq r$. In this case, 
\[\hat{\puRational}\left(\frac{1}{k},\dots, \frac{1}{k},0,\dots, 0\right)=k^{m-1}.
\]
Furthermore, since $m\geq 1$, the maximum happens when $k=r$. Therefore, 
\begin{equation}\label{eq:spherical-part}
    \EE_\theta [\puRational(\theta)]
    \;\leq\;
    \sup_{\theta\in \mathbb{S}^{2r-1}}\puRational(\theta)
    \;=\;
    \sup_{X\in \Delta^{r-1}}\hat{\puRational}(X)
    \;=\;
    r^{m-1}.
\end{equation}

Combining \eqref{eq:radial-part} and \eqref{eq:spherical-part}, we have $$\int_{V(\lambda_{1,0}^D+\cdots + \lambda_{r,0}^D)} d'{\lambdazeros}q({\lambdazeros})\leq \prod_{j=2}^m\frac{r}{r-j} \leq\prod_{j=2}^D\frac{r}{r-j} $$
when $r>D$. We denote $\prod_{j=2}^D\frac{r}{r-j}$ by $R_{r,D}$.

Hence, we obtain the following bound: 
\begin{equation}\label{eq:ME-bound-RD}\mainExpectation\\\leq \frac{\pi R_{r,D}}{nD^2}{D \choose k}^2\sum_{m=1}^k {k\choose m}\frac{(m+n-1)!}{(n-1)!}.
\end{equation}

\subsection{Consequences}\label{ss:consequences}

The previous subsections have derived  \eqref{eq:ME-bound-RD}. In this subsection, we discuss its consequences. In particular we use \eqref{eq:ME-bound-RD} in the proof of our main result, and see the effects of Waring length in our complexity results.

\begin{theorem}\label{thm:waring-condition-bound}
Fix integers $n\geq 1$, $D\geq 2$, and $r>D$.
Let $f_L$ be a Gaussian Waring representation of length $r$. Then,
\[
\EE_L\big[\GammaFrob(f_L)^2\big]\leq 
\frac{\pi }{4} R_{r,D} (D-1)^3
n\left( 1+\frac{3}{n}\right)^2
\]
where $R_{r,D}:=\prod\limits_{j=2}^D\frac{r}{r-j}$.
\end{theorem}

\begin{proof}
Recall that \eqref{eq:ME-bound-RD} gives the bound
\begin{equation*}
\mainExpectation\\\leq \frac{\pi R_{r,D}}{nD^2}{D \choose k}^2\sum_{m=1}^k {k\choose m}\frac{(m+n-1)!}{(n-1)!}.
\end{equation*}
Since 
\[
 \sum_{m=1}^k {k \choose m}\frac{(m+n-1)!}{(n-1)!}
\;\leq\;
(n+k+1)^k,
\]
we have
\[\mainExpectation\leq R_{r,D}\cdot \frac{\pi }{nD^2}{D \choose k}^2
(n+k+1)^k.
\]
Motivated by this bound, define
\[
M_k := \frac{\pi}{nD^2}{D \choose k}^2 (n+k+1)^k,
\qquad k \ge 2.
\]
so  the expected squared condition number satisfies 
\[
\EE_L\big[\GammaFrob(f_L)^2\big]
\le \sum_{k=2}^D \left(R_{r,D} M_k\right)^{\frac{1}{k-1}}.
\]

Next, we bound $(M_k)^{\frac{1}{k-1}}$. Note that
$\{(M_k)^{\frac{1}{k-1}}\}_{k \ge 2}$ attains its maximum at $k=2$. Indeed,
\[
M_2
= \frac{\pi(D-1)^2(n+3)^2}{4n}.
\]
Hence, $(M_k)^{\frac{1}{k-1}} \leq M_2$ for all $k \geq 2$.

Since $R_{r,D} > 1$, we also have $R_{r,D}^{\frac{1}{k-1}} \leq R_{r,D}$ for all $k \geq 2$. Combining these estimates yields
\begin{align*}
\EE_L\big[\GammaFrob(f_L)^2\big]
\;\le\; \sum_{k=2}^D \left(R_{r,D} M_k\right)^{\frac{1}{k-1}} 
\;\le\; R_{r,D} \sum_{k=2}^D M_2 
\;\le\; R_{r,D} \frac{\pi(D-1)^3 (n+3)^2}{4n}.
\end{align*}
Finally, rewriting $(n+3)^2 = n^2 \left(1+\frac{3}{n}\right)^2$
gives the stated bound.
\end{proof}    

We note the impacts of $R_{r,D}$. The factor  $R_{r,D}$ captures the additional cost induced by the Waring structure. Since $R_{r,D} = 1 + O(\frac{1}{r})$ as $r \to \infty$ (for fixed $D$), this structural penalty vanishes when $r \gg D$. In particular, for sufficiently large $r$, the expected condition bound exhibits polynomial dependence on $n$ and $D$.

We extend this result to homogeneous systems that consist of Waring representations. Let $f=(f_{L^{(1)}},\dots, f_{L^{(n)}})\in H_{d_1}\times\cdots \times H_{d_n}$ be a homogeneous system defined with $L^{(i)}=(L^{(i)}_1,\dots, L^{(i)}_{r_i})$ for $i=1,\dots, n$. We define $D=\max\{d_1,\dots, d_n\}, R=\max_{i}\{R_{r_i,d_i}\}$ and
\[
\Gamma_W :=
\frac{\sqrt{\pi}}{2}\sqrt{R}(D-1)^{\frac{3}{2}}
n\left(1+\frac{3}{n}\right).
\]
$\Gamma_W$ induces an upper bound on the averaged condition number of the system $f$.

Then, the following provides the complexity of rigid homotopy for finding a root of the system $f$:
\begin{corollary}\label{cor:system}
Let $f=(f_{L^{(1)}},\dots,f_{L^{(n)}})\in H_{d_1}\times\cdots \times H_{d_n}$ be a homogeneous polynomial system of Waring representations such that each $f_{L^{(i)}}$ is independent Gaussian random polynomials with Waring representation of length $r_i$ satisfying $r_i>d_i$ and degree at most $D$. Let $E$ be an upper bound on the evaluation complexity of $f$. Then, \Cref{alg:rigid_continuation} on input $f$ and $\epsilon>0$, terminates after
\[
E\cdot p(n,D)\left(\Gamma_W\log \Gamma_W+\log\log \frac{1}{\epsilon}\right)
\]
operations on average for some polynomial $p$ in variables $n$ and $D$.
\end{corollary}
\begin{proof}
    Recall that for each $i$, we have
    \[\EE_{L^{(i)}}\big[\GammaFrob(f_{L^{(i)}})^2\big]\leq 
\frac{\pi}{4} R_{r_i,d_i} (d_i-1)^3
n\left( 1+\frac{3}{n}\right)^2\]
by \Cref{thm:waring-condition-bound}.
Note that 
\begin{equation*}
\EE_{\{L^{(i)}\}}[\GammaFrob(f)^2]=\EE_{\{L^{(i)}\}}\left[\sum_{i=1}^n\GammaFrob(f_{L^{(i)}})^2\right]    
\end{equation*}
by \eqref{eq:averaged-gamma-system}. Then, we have
\begin{align*}
\EE_{\{L^{(i)}\}}\left[\sum_{i=1}^n\GammaFrob(f_{L^{(i)}})^2\right] &=\sum_{i=1}^n\EE_{L^{(i)}}[\GammaFrob(f_{L^{(i)}})^2]\\&\leq\sum_{i=1}^n  \frac{\pi}{4} R_{r_i,d_i} (d_i-1)^3
n\left( 1+\frac{3}{n}\right)^2\\
&\leq n\cdot \frac{\pi}{4} R (D-1)^3
n\left( 1+\frac{3}{n}\right)^2
\end{align*}
Therefore, we have 
\begin{equation*}
\EE_{\{L^{(i)}\}}\left[\GammaFrob(f)^2\right]^{\frac{1}{2}}\leq \Gamma_W.
\end{equation*}
Applying \cite[Corollary 1.3]{BCL2023-rigid-two}, the result follows.
\end{proof}

\section{Computational results for Waring systems}\label{s:comp-results}

In this section, we conduct a series of experiments to examine the performance and behavior of rigid homotopy, with a particular focus on how the $\gamma$-based estimates relate to step size and path geometry. \Cref{ss:cr-random-systems} presents performance on random Waring systems, summarized in \Cref{tab:ndr-metrics}. \Cref{ss:estimates-empirical} studies the empirical behavior of $\GammaEstimate$ for individual instances. In \Cref{ss:cr-visualizing-cool-systems}, we visualize solution paths to illustrate how the $\gamma$-number varies along the trajectory. \Cref{ss:cr-heuristics} compares \Cref{alg:rigid_continuation} with continuation with heuristic step sizes, demonstrating that speedups are often possible without loss of accuracy. Finally, \Cref{ss:scalability-considerations} discusses limitations of the current approach and implications for scalability.

For all of our computational results, we use our implementation of \Cref{alg:rigid_sampling,alg:gamma-estimate,alg:rigid_continuation} in \texttt{Julia}. The implementation uses \texttt{Enzyme} \cite{NEURIPS2020_9332c513,10.1145/3458817.3476165} for automatic differentiation of polynomials and \texttt{FFTW} \cite{FFTW05} for computing the degree $k$ parts of polynomials in \Cref{alg:gamma-estimate}. Throughout the experiments, we use the constant $\epsilon = 10^{-8}$ and $K=10^6$ for $n=1$ and $K=10^7$ for $n=2$ for \Cref{alg:rigid_continuation}. The code is available~at 

\begin{center}
\url{https://github.com/abigailrjones/rigid-homotopies}\,.
\end{center}

\subsection{Performance of rigid homotopy on average}\label{ss:cr-random-systems}

We investigate the behavior of the rigid homotopy algorithm on random polynomial systems with Waring representations. 
For each triple $(n,D,r)$, we generate random Waring systems with length $r$ and degree $D$ polynomials, and apply the continuation. For each configuration, we record the number of iterations, average step size, the average $\GammaEstimate$, and the average condition number \eqref{eq:condition-number} along the path. The reported values are averages over 100 random trials. Results are summarized in \Cref{tab:ndr-metrics}.

\begin{table}[t]
    \centering
    \small\renewcommand{\arraystretch}{1.1}
\noindent\makebox[\textwidth]{
    \begin{tabular}{l c c c c c c c c}
        \toprule
        & \multicolumn{2}{c}{Iterations} & \multicolumn{2}{c}{\begin{tabular}[c]{@{}c@{}}Average  \\ step size\end{tabular}} & \multicolumn{2}{c}{\begin{tabular}[c]{@{}c@{}}Average  \\ $\GammaEstimate$\end{tabular}} & \multicolumn{2}{c}{\begin{tabular}[c]{@{}c@{}}Average \\ condition number\end{tabular}} \\
        \cmidrule(lr){2-3} \cmidrule(lr){4-5} \cmidrule(lr){6-7} \cmidrule(lr){8-9}
        $(n, D, r)$ & Mean & Median & Mean & Median & Mean & Median & Mean & Median \\
        \midrule
        $(1, 3, 4)$ & $3.19 \times 10^{4}$ & $2.88 \times 10^{4}$ & $3.37 \times 10^{-5}$ & $3.47 \times 10^{-5}$ & $133.88$ & $120.67$ & $1.0$ & $1.0$ \\
        $(1, 3, 5)$ & $3.05 \times 10^{4}$ & $2.69 \times 10^{4}$ & $3.53 \times 10^{-5}$ & $3.71 \times 10^{-5}$ & $127.78$ & $112.84$ & $1.0$ & $1.0$ \\
        $(1, 3, 6)$ & $2.91 \times 10^{4}$ & $2.62 \times 10^{4}$ & $3.63 \times 10^{-5}$ & $3.81 \times 10^{-5}$ & $121.89$ & $109.94$ & $1.0$ & $1.0$ \\
        $(1, 4, 5)$ & $4.27 \times 10^{4}$ & $3.98 \times 10^{4}$ & $2.46 \times 10^{-5}$ & $2.51 \times 10^{-5}$ & $178.87$ & $166.86$ & $1.0$ & $1.0$ \\
        $(1, 4, 6)$ & $4.39 \times 10^{4}$ & $3.97 \times 10^{4}$ & $2.40 \times 10^{-5}$ & $2.52 \times 10^{-5}$ & $184.0$ & $166.44$ & $1.0$ & $1.0$ \\
        $(1, 4, 7)$ & $4.05 \times 10^{4}$ & $3.77 \times 10^{4}$ & $2.54 \times 10^{-5}$ & $2.66 \times 10^{-5}$ & $169.8$ & $157.68$ & $1.0$ & $1.0$ \\
        $(1, 5, 6)$ & $5.87 \times 10^{4}$ & $5.20 \times 10^{4}$ & $1.81 \times 10^{-5}$ & $1.92 \times 10^{-5}$ & $245.86$ & $217.94$ & $1.0$ & $1.0$ \\
        $(1, 5, 7)$ & $5.86 \times 10^{4}$ & $5.39 \times 10^{4}$ & $1.81 \times 10^{-5}$ & $1.85 \times 10^{-5}$ & $245.48$ & $225.85$ & $1.0$ & $1.0$ \\
        $(1, 5, 8)$ & $5.53 \times 10^{4}$ & $5.18 \times 10^{4}$ & $1.88 \times 10^{-5}$ & $1.93 \times 10^{-5}$ & $231.6$ & $216.81$ & $1.0$ & $1.0$ \\
        $(2, 3, 4)$ & $3.37 \times 10^{5}$ & $2.76 \times 10^{5}$ & $3.68 \times 10^{-6}$ & $3.62 \times 10^{-6}$ & $459.92$ & $448.21$ & $2.37$ & $1.79$ \\
        $(2, 3, 5)$ & $3.16 \times 10^{5}$ & $2.57 \times 10^{5}$ & $4.00 \times 10^{-6}$ & $3.89 \times 10^{-6}$ & $447.74$ & $418.57$ & $2.05$ & $1.75$ \\
        $(2, 3, 6)$ & $3.32 \times 10^{5}$ & $2.71 \times 10^{5}$ & $3.81 \times 10^{-6}$ & $3.69 \times 10^{-6}$ & $408.03$ & $395.88$ & $2.21$ & $1.97$ \\
        $(2, 4, 5)$ & $5.01 \times 10^{5}$ & $4.38 \times 10^{5}$ & $2.40 \times 10^{-6}$ & $2.28 \times 10^{-6}$ & $617.64$ & $599.98$ & $2.4$ & $1.98$ \\
        $(2, 4, 6)$ & $4.44 \times 10^{5}$ & $3.83 \times 10^{5}$ & $2.68 \times 10^{-6}$ & $2.61 \times 10^{-6}$ & $597.94$ & $579.36$ & $2.43$ & $1.86$ \\
        $(2, 4, 7)$ & $4.18 \times 10^{5}$ & $3.46 \times 10^{5}$ & $2.94 \times 10^{-6}$ & $2.89 \times 10^{-6}$ & $581.29$ & $564.7$ & $2.13$ & $1.7$ \\
        $(2, 5, 6)$ & $6.58 \times 10^{5}$ & $5.53 \times 10^{5}$ & $1.89 \times 10^{-6}$ & $1.81 \times 10^{-6}$ & $779.84$ & $748.32$ & $2.55$ & $1.96$ \\
        $(2, 5, 7)$ & $6.35 \times 10^{5}$ & $5.16 \times 10^{5}$ & $2.06 \times 10^{-6}$ & $1.94 \times 10^{-6}$ & $748.54$ & $738.31$ & $2.36$ & $1.99$ \\
        $(2, 5, 8)$ & $5.74 \times 10^{5}$ & $5.20 \times 10^{5}$ & $2.01 \times 10^{-6}$ & $1.92 \times 10^{-6}$ & $746.77$ & $730.82$ & $2.28$ & $2.12$ \\
        \bottomrule
    \end{tabular}}
    \caption{Summary of metrics (iterations, step size, $\Gamma_{\mathrm{Estimate}}$, and condition number) for different $(n, D, r)$ configurations. Values are obtained by averaging over $100$ random systems for each configuration.    }
    \label{tab:ndr-metrics}
\end{table}

We read several trends from the data. The experimental results reveal a correlation between the step size and the estimated $\gamma$-numbers as expected. In particular, smaller step sizes correspond to the inverse dependence of the step size on the $\gamma$-number.

Comparing the cases $n=1$ and $n=2$, we observe an increase in the estimated $\Gamma$-values and iteration counts, reflecting the increased geometric complexity of higher-dimensional~systems. In contrast, the dependence on $r$ appears weak in these experiments, with no systematic growth as $r$ increases. This is consistent with the theoretical prediction that the dependence on $r$ decreases and becomes asymptotically negligible.

Finally, a few discrepancies between the mean and median values indicate the presence of occasional large outliers, suggesting that rare ill-conditioned instances may affect the algorithm's practical behavior. This is explored more in depth in \Cref{ss:estimates-empirical}.

\subsection{Estimates of $\gammaFrob$-number}\label{ss:estimates-empirical}

We investigate the behavior of \Cref{alg:gamma-estimate} with a series of numerical experiments. Fixing the degree $D$ with variables $z=(z_0,\dots, z_n)$, we vary the Waring representation length $r$ from $D+1$ to $100$. For each value of $r$, we generate $20$ random polynomials $f$ with a Waring representation of length $r$. For each such polynomial, we sample $20$ approximate roots $\zeta$ of~$f$.

For each pair $(f,\zeta)$, we compute the estimator $\GammaEstimate$ of $\gamma_{\mathrm{Frob}}(f,\zeta)$ using \Cref{alg:gamma-estimate}, and then take the square. Averaging these values over the sampled polynomials and roots, we observe an empirical behavior of $\GammaEstimate^2$ for polynomials with fixed $n$ and $D$. We perform this experiment for degrees $D=4$ and $D=5$. The results are shown in \Cref{fig:gamma_estimate2}, where the vertical axis displays the estimated average of $\GammaEstimate^2$ for each value of $r$.

\begin{figure}[!htbp]
    \centering
    \begin{subfigure}[b]{.89\textwidth}
        \centering
        \resizebox{\linewidth}{!}{\begin{tikzpicture}[scale=.7]
\begin{axis}[
    width=10cm, height=7cm,
    xlabel={$r$},
    ylabel={Average of $\GammaEstimate^2$},
    grid=major,
    legend pos=north east
]
\addplot[mark=*, blue, only marks, opacity=1] coordinates {
(5, 645.5617) (6, 5447.8733) (7, 1858.9247) (8, 5963.7578) (9, 675.4258) (10, 2118.5733) (11, 2907.7372) (12, 2343.647) (13, 64320.6808) (14, 477.7228) (15, 13482.1501) (16, 947.3326) (17, 450.822) (18, 11816.5683) (19, 35118.4783) (20, 1397.5333) (21, 466.5037) (22, 401.0752) (23, 362.1123) (24, 827.1659) (25, 372.4582) (26, 430.1983) (27, 283.9326) (28, 1047.383) (29, 1017.1835) (30, 448.9853) (31, 436.9891) (32, 348.7851) (33, 305.8139) (34, 1747.6618) (35, 918.7761) (36, 1371.8469) (37, 377.1273) (38, 330.2009) (39, 569.7187) (40, 409.5535) (41, 951.6676) (42, 371.8632) (43, 347.1887) (44, 1227.1964) (45, 427.5645) (46, 674.6588) (47, 297.4155) (48, 306.3453) (49, 940.249) (50, 361.3502) (51, 557.3659) (52, 311.1917) (53, 426.4473) (54, 34022.0123) (55, 328.8204) (56, 388.5177) (57, 458.3319) (58, 501.4842) (59, 336.7822) (60, 769.8601) (61, 425.7506) (62, 3293.1558) (63, 482.2235) (64, 402.3625) (65, 1321.5137) (66, 431.3051) (67, 878.9998) (68, 1024.9186) (69, 329.267) (70, 4603.5983) (71, 365.566) (72, 295.6534) (73, 332.9381) (74, 19188.8077) (75, 13592.3924) (76, 369.4175) (77, 1443.6717) (78, 879.8735) (79, 1179.181) (80, 12897.0695) (81, 394.846) (82, 467.8329) (83, 143966.0735) (84, 314.9919) (85, 416.6155) (86, 429.6865) (87, 450.3552) (88, 1526.5817) (89, 1426.2998) (90, 316.1106) (91, 378.7713) (92, 70276.3191) (93, 357.493) (94, 682.3093) (95, 316.3758) (96, 344.4638) (97, 451.5896) (98, 289.6763) (99, 402.4046) (100, 389.9495) 
};

\end{axis}

\end{tikzpicture}\hspace{1em}\begin{tikzpicture}[scale=.7]
\begin{axis}[
    width=10cm, height=7cm,
    xlabel={$r$},
    ylabel={Average of $\GammaEstimate^2$},
    grid=major,
    legend pos=north east
]
\addplot[mark=*, blue, only marks, opacity=1] coordinates {
(5, 645.5617) (6, 5447.8733) (7, 1858.9247) (8, 5963.7578) (9, 675.4258) (10, 2118.5733) (11, 2907.7372) (12, 2343.647) (14, 477.7228) (16, 947.3326) (17, 450.822) (20, 1397.5333) (21, 466.5037) (22, 401.0752) (23, 362.1123) (24, 827.1659) (25, 372.4582) (26, 430.1983) (27, 283.9326) (28, 1047.383) (29, 1017.1835) (30, 448.9853) (31, 436.9891) (32, 348.7851) (33, 305.8139) (34, 1747.6618) (35, 918.7761) (36, 1371.8469) (37, 377.1273) (38, 330.2009) (39, 569.7187) (40, 409.5535) (41, 951.6676) (42, 371.8632) (43, 347.1887) (44, 1227.1964) (45, 427.5645) (46, 674.6588) (47, 297.4155) (48, 306.3453) (49, 940.249) (50, 361.3502) (51, 557.3659) (52, 311.1917) (53, 426.4473) (55, 328.8204) (56, 388.5177) (57, 458.3319) (58, 501.4842) (59, 336.7822) (60, 769.8601) (61, 425.7506) (62, 3293.1558) (63, 482.2235) (64, 402.3625) (65, 1321.5137) (66, 431.3051) (67, 878.9998) (68, 1024.9186) (69, 329.267) (70, 4603.5983) (71, 365.566) (72, 295.6534) (73, 332.9381) (76, 369.4175) (77, 1443.6717) (78, 879.8735) (79, 1179.181) (81, 394.846) (82, 467.8329) (84, 314.9919) (85, 416.6155) (86, 429.6865) (87, 450.3552) (88, 1526.5817) (89, 1426.2998) (90, 316.1106) (91, 378.7713) (93, 357.493) (94, 682.3093) (95, 316.3758) (96, 344.4638) (97, 451.5896) (98, 289.6763) (99, 402.4046) (100, 389.9495)  
};

\end{axis}

\end{tikzpicture}}
        \caption{$n=2, D=4$}
    \end{subfigure}
    \begin{subfigure}[b]{.89\textwidth}
        \centering
        \resizebox{\linewidth}{!}{\begin{tikzpicture}[scale=.7]
\begin{axis}[
    width=10cm, height=7cm,
    xlabel={$r$},
    ylabel={Average of $\GammaEstimate^2$},
    grid=major,
    legend pos=north east
]
\addplot[mark=*, blue, only marks, opacity=1] coordinates {
(6, 4096.4296) (7, 3497.6448) (8, 809.3147) (9, 7453.34) (10, 81992.2864) (11, 2068.0461) (12, 10172.9752) (13, 5855.7671) (14, 330352.5368) (15, 521.6094) (16, 1717.2093) (17, 3896.9887) (18, 682.729) (19, 63155.6276) (20, 1196.5814) (21, 1079.4544) (22, 764.3417) (23, 805.833) (24, 2722.9913) (25, 1564.1797) (26, 4015.8811) (27, 31313.3861) (28, 2318.1202) (29, 772137.8244) (30, 998.3947) (31, 434887.9841) (32, 786.9879) (33, 2274.9024) (34, 2863.5081) (35, 1556.1012) (36, 535.6399) (37, 28234.1718) (38, 694.0526) (39, 3.717619281256046e11) (40, 1249.2934) (41, 1222.5574) (42, 639.1423) (43, 7581.0992) (44, 12146.596) (45, 1638.7763) (46, 1832.3996) (47, 1127.7228) (48, 956.8939) (49, 3861.9978) (50, 43611.1613) (51, 1643.9786) (52, 5659.0326) (53, 639.8823) (54, 1071.8346) (55, 1143.3907) (56, 26508.7395) (57, 462.2703) (58, 8929.2893) (59, 28139.3257) (60, 1275.7047) (61, 745.6566) (62, 1167.8861) (63, 620.7622) (64, 2515.3741) (65, 10656.2711) (66, 2061.4618) (67, 17158.8158) (68, 478.8959) (69, 5398.3491) (70, 787.9574) (71, 688.2826) (72, 7133.232) (73, 909.4528) (74, 9152.3197) (75, 513.0483) (76, 1779.8825) (77, 3245.8616) (78, 541.608) (79, 887.8417) (80, 537.6382) (81, 2867.4008) (82, 2140.8077) (83, 3091.5186) (84, 3420.5619) (85, 96274.0227) (86, 987.4844) (87, 1044.0125) (88, 2066.4705) (89, 1420.1269) (90, 2757.7723) (91, 22192.2788) (92, 3455.9532) (93, 6735.9005) (94, 562.8469) (95, 5871.0994) (96, 552.5235) (97, 1397.8718) (98, 980.3407) (99, 13697.847) (100, 1096.8843) 
};

\end{axis}

\end{tikzpicture}\hspace{1em}\begin{tikzpicture}[scale=.7]
\begin{axis}[
    width=10cm, height=7cm,
    xlabel={$r$},
    ylabel={Average of $\GammaEstimate^2$},
    grid=major,
    legend pos=north east
]
\addplot[mark=*, blue, only marks, opacity=1] coordinates {
(6, 4096.4296) (7, 3497.6448) (8, 809.3147) (9, 7453.34) (11, 2068.0461) (12, 10172.9752) (13, 5855.7671) (15, 521.6094) (16, 1717.2093) (17, 3896.9887) (18, 682.729) (20, 1196.5814) (21, 1079.4544) (22, 764.3417) (23, 805.833) (24, 2722.9913) (25, 1564.1797) (26, 4015.8811) (28, 2318.1202) (30, 998.3947) (32, 786.9879) (33, 2274.9024) (34, 2863.5081) (35, 1556.1012) (36, 535.6399) (38, 694.0526) (40, 1249.2934) (41, 1222.5574) (42, 639.1423) (43, 7581.0992) (44, 12146.596) (45, 1638.7763) (46, 1832.3996) (47, 1127.7228) (48, 956.8939) (49, 3861.9978) (51, 1643.9786) (52, 5659.0326) (53, 639.8823) (54, 1071.8346) (55, 1143.3907) (56, 26508.7395) (57, 462.2703) (58, 8929.2893) (59, 28139.3257) (60, 1275.7047) (61, 745.6566) (62, 1167.8861) (63, 620.7622) (64, 2515.3741) (65, 10656.2711) (66, 2061.4618) (67, 17158.8158) (68, 478.8959) (69, 5398.3491) (70, 787.9574) (71, 688.2826) (72, 7133.232) (73, 909.4528) (74, 9152.3197) (75, 513.0483) (76, 1779.8825) (77, 3245.8616) (78, 541.608) (79, 887.8417) (80, 537.6382) (81, 2867.4008) (82, 2140.8077) (83, 3091.5186) (84, 3420.5619) (86, 987.4844) (87, 1044.0125) (88, 2066.4705) (89, 1420.1269) (90, 2757.7723) (91, 22192.2788) (92, 3455.9532) (93, 6735.9005) (94, 562.8469) (95, 5871.0994) (96, 552.5235) (97, 1397.8718) (98, 980.3407) (99, 13697.847) (100, 1096.8843) 
};

\end{axis}

\end{tikzpicture}}
        \caption{$n=2, D=5$}
    \end{subfigure}
    \begin{subfigure}[b]{.89\textwidth}
        \centering
        \resizebox{\linewidth}{!}{\begin{tikzpicture}[scale=.7]
\begin{axis}[
    width=10cm, height=7cm,
    xlabel={$r$},
    ylabel={Average of $\GammaEstimate^2$},
    grid=major,
    legend pos=north east
]
\addplot[mark=*, blue, only marks, opacity=1] coordinates {
(5, 3492.7867) (6, 5098.9121) (7, 3324.8629) (8, 2695.7196) (9, 3499.6569) (10, 2249.1798) (11, 3139.4922) (12, 3529.1721) (13, 8505.7678) (14, 2210.626) (15, 3958.1048) (16, 2404.7047) (17, 2206.0547) (18, 25326.132) (19, 2163.6986) (20, 5249.6657) (21, 2595.7797) (22, 3495.5804) (23, 1990.4671) (24, 2415.1105) (25, 3481.1528) (26, 2495.2656) (27, 1926.9881) (28, 9775.0114) (29, 2500.776) (30, 2713.19) (31, 1907.6926) (32, 2066.9946) (33, 2458.5405) (34, 3022.3854) (35, 3523.2971) (36, 4040.8763) (37, 2981.6754) (38, 1897.5855) (39, 2292.0795) (40, 2007.489) (41, 2229.9708) (42, 3553.7796) (43, 2046.471) (44, 2480.7478) (45, 1950.1337) (46, 2289.6671) (47, 2078.8643) (48, 1608.6021) (49, 1944.9722) (50, 2371.1377) (51, 1634.9037) (52, 1748.398) (53, 4663.1207) (54, 3722.267) (55, 1969.2126) (56, 2284.5241) (57, 5791.24) (58, 2098.1812) (59, 2766.316) (60, 1828.8115) (61, 2188.1527) (62, 3385.2218) (63, 3525.6189) (64, 3175.7065) (65, 1821.021) (66, 3756.1825) (67, 2021.2069) (68, 4844.698) (69, 1985.763) (70, 1885.282) (71, 2009.116) (72, 1837.5564) (73, 1909.117) (74, 1783.4092) (75, 2944.5011) (76, 2481.2603) (77, 2081.7423) (78, 3010.36) (79, 1883.3174) (80, 2665.8259) (81, 1943.167) (82, 3926.8173) (83, 3455.1892) (84, 1996.4519) (85, 1664.8973) (86, 2743.4612) (87, 2907.8051) (88, 1939.9131) (89, 1874.5624) (90, 2088.7911) (91, 2166.4928) (92, 2256.4287) (93, 1979.5113) (94, 2804.0193) (95, 3450.8725) (96, 2650.5937) (97, 2128.6628) (98, 2154.2914) (99, 3641.9359) (100, 1763.1572) 
};

\end{axis}

\end{tikzpicture}\hspace{1em}\begin{tikzpicture}[scale=.7]
\begin{axis}[
    width=10cm, height=7cm,
    xlabel={$r$},
    ylabel={Average of $\GammaEstimate^2$},
    grid=major,
    legend pos=north east
]
\addplot[mark=*, blue, only marks, opacity=1] coordinates {
(5, 3492.7867) (7, 3324.8629) (8, 2695.7196) (9, 3499.6569) (10, 2249.1798) (11, 3139.4922) (12, 3529.1721) (14, 2210.626) (16, 2404.7047) (17, 2206.0547) (19, 2163.6986) (21, 2595.7797) (22, 3495.5804) (23, 1990.4671) (24, 2415.1105) (25, 3481.1528) (26, 2495.2656) (27, 1926.9881) (29, 2500.776) (30, 2713.19) (31, 1907.6926) (32, 2066.9946) (33, 2458.5405) (34, 3022.3854) (35, 3523.2971) (37, 2981.6754) (38, 1897.5855) (39, 2292.0795) (40, 2007.489) (41, 2229.9708) (43, 2046.471) (44, 2480.7478) (45, 1950.1337) (46, 2289.6671) (47, 2078.8643) (48, 1608.6021) (49, 1944.9722) (50, 2371.1377) (51, 1634.9037) (52, 1748.398) (55, 1969.2126) (56, 2284.5241) (58, 2098.1812) (59, 2766.316) (60, 1828.8115) (61, 2188.1527) (62, 3385.2218) (63, 3525.6189) (64, 3175.7065) (65, 1821.021) (67, 2021.2069) (69, 1985.763) (70, 1885.282) (71, 2009.116) (72, 1837.5564) (73, 1909.117) (74, 1783.4092) (75, 2944.5011) (76, 2481.2603) (77, 2081.7423) (78, 3010.36) (79, 1883.3174) (80, 2665.8259) (81, 1943.167) (83, 3455.1892) (84, 1996.4519) (85, 1664.8973) (86, 2743.4612) (87, 2907.8051) (88, 1939.9131) (89, 1874.5624) (90, 2088.7911) (91, 2166.4928) (92, 2256.4287) (93, 1979.5113) (94, 2804.0193) (95, 3450.8725) (96, 2650.5937) (97, 2128.6628) (98, 2154.2914) (100, 1763.1572) 
};

\end{axis}

\end{tikzpicture}}
        \caption{$n=3, D=4$}
    \end{subfigure}
    \begin{subfigure}[b]{.89\textwidth}
        \centering
        \resizebox{\linewidth}{!}{\begin{tikzpicture}[scale=.7]
\begin{axis}[
    width=10cm, height=7cm,
    xlabel={$r$},
    ylabel={Average of $\GammaEstimate^2$},
    grid=major,
    legend pos=north east
]
\addplot[mark=*, blue, only marks, opacity=1] coordinates {
(6, 4489.8812) (7, 8761.5054) (8, 7563.794) (9, 10641.5622) (10, 8183.8476) (11, 8145.4969) (12, 5385.3804) (13, 5895.3976) (14, 4580.5721) (15, 5288.5137) (16, 3687.8464) (17, 7536.4082) (18, 24765.203) (19, 139641.2089) (20, 3984.9352) (21, 6855.9719) (22, 5816.4307) (23, 4594.3515) (24, 5439.6665) (25, 4697.5229) (26, 5523.3899) (27, 3536.9897) (28, 5774.2787) (29, 4638.9843) (30, 9893.6559) (31, 4458.4331) (32, 4886.8664) (33, 6895.1911) (34, 11153.1402) (35, 6146.0975) (36, 6774.1936) (37, 6414.9045) (38, 4941.7717) (39, 15808.6995) (40, 6585.7939) (41, 8687.7372) (42, 4945.0997) (43, 7182.6157) (44, 8114.9807) (45, 9373.7454) (46, 9622.3865) (47, 8220.3901) (48, 7679.1713) (49, 21552.8463) (50, 5719.1321) (51, 3724.6796) (52, 9075.7085) (53, 5322.0356) (54, 4668.298) (55, 4388.875) (56, 4949.4502) (57, 5344.0218) (58, 3845.6227) (59, 6046.7662) (60, 11267.3257) (61, 5816.1952) (62, 4088.9705) (63, 6181.1906) (64, 4313.1368) (65, 4927.7864) (66, 7369.1977) (67, 5138.7677) (68, 7138.8928) (69, 9481.7659) (70, 3202.6075) (71, 4172.9391) (72, 7288.7945) (73, 6789.9229) (74, 3350.7092) (75, 4188.4287) (76, 4373.5609) (77, 4180.3348) (78, 4417.9113) (79, 4954.7796) (80, 5946.8569) (81, 3670.2975) (82, 3002.1782) (83, 10668.3833) (84, 9192.4597) (85, 4447.8027) (86, 4923.3719) (87, 5129.4287) (88, 3308.6746) (89, 3925.31) (90, 3193.0574) (91, 4103.3762) (92, 14004.8268) (93, 4189.5285) (94, 4143.2355) (95, 5411.6321) (96, 4374.7236) (97, 4019.2449) (98, 6097.4317) (99, 3910.3566) (100, 4168.1945) 
};

\end{axis}

\end{tikzpicture}\hspace{1em}\begin{tikzpicture}[scale=.7]
\begin{axis}[
    width=10cm, height=7cm,
    xlabel={$r$},
    ylabel={Average of $\GammaEstimate^2$},
    grid=major,
    legend pos=north east
]
\addplot[mark=*, blue, only marks, opacity=1] coordinates {
(6, 4489.8812) (7, 8761.5054) (8, 7563.794) (10, 8183.8476) (11, 8145.4969) (12, 5385.3804) (13, 5895.3976) (14, 4580.5721) (15, 5288.5137) (16, 3687.8464) (17, 7536.4082) (20, 3984.9352) (21, 6855.9719) (22, 5816.4307) (23, 4594.3515) (24, 5439.6665) (25, 4697.5229) (26, 5523.3899) (27, 3536.9897) (28, 5774.2787) (29, 4638.9843) (31, 4458.4331) (32, 4886.8664) (33, 6895.1911) (35, 6146.0975) (36, 6774.1936) (37, 6414.9045) (38, 4941.7717) (40, 6585.7939) (41, 8687.7372) (42, 4945.0997) (43, 7182.6157) (44, 8114.9807) (47, 8220.3901) (48, 7679.1713) (50, 5719.1321) (51, 3724.6796) (53, 5322.0356) (54, 4668.298) (55, 4388.875) (56, 4949.4502) (57, 5344.0218) (58, 3845.6227) (59, 6046.7662) (61, 5816.1952) (62, 4088.9705) (63, 6181.1906) (64, 4313.1368) (65, 4927.7864) (66, 7369.1977) (67, 5138.7677) (68, 7138.8928) (70, 3202.6075) (71, 4172.9391) (72, 7288.7945) (73, 6789.9229) (74, 3350.7092) (75, 4188.4287) (76, 4373.5609) (77, 4180.3348) (78, 4417.9113) (79, 4954.7796) (80, 5946.8569) (81, 3670.2975) (82, 3002.1782) (85, 4447.8027) (86, 4923.3719) (87, 5129.4287) (88, 3308.6746) (89, 3925.31) (90, 3193.0574) (91, 4103.3762) (93, 4189.5285) (94, 4143.2355) (95, 5411.6321) (96, 4374.7236) (97, 4019.2449) (98, 6097.4317) (99, 3910.3566) (100, 4168.1945) 
};

\end{axis}

\end{tikzpicture}}
        \caption{$n=3, D=5$}
    \end{subfigure}
    \caption{Estimates of $\gammaFrob^2(f,\zeta)$ for fixed $(n,D)$. Left: raw data. Right: the same data with the largest $15$ values removed.}
    \label{fig:gamma_estimate2}
\end{figure}
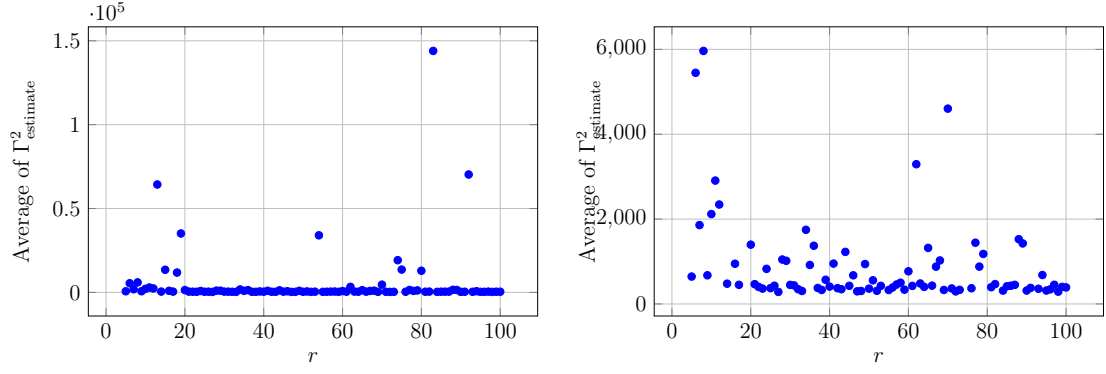
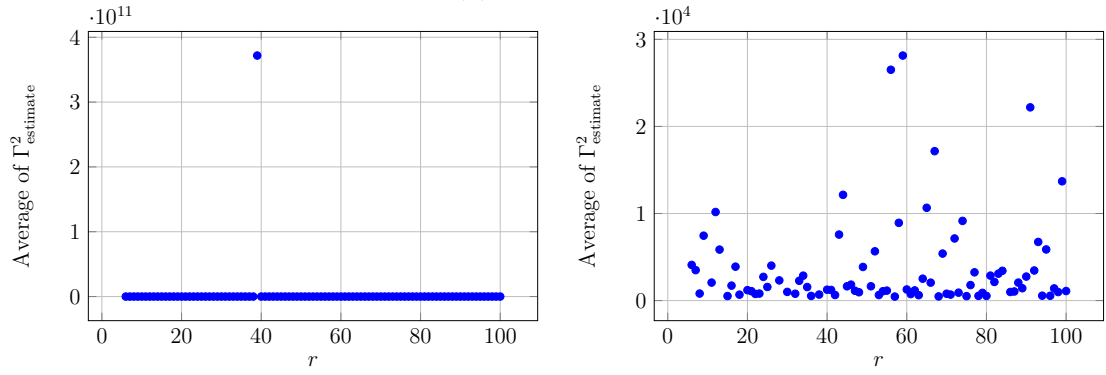
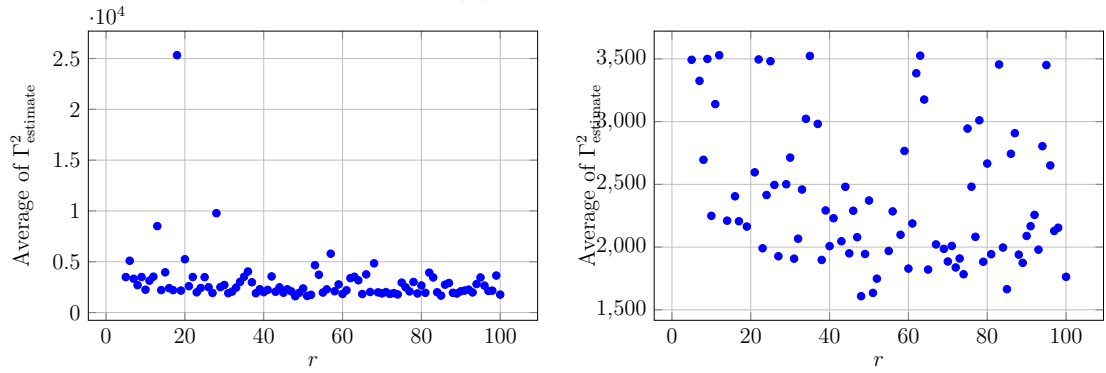
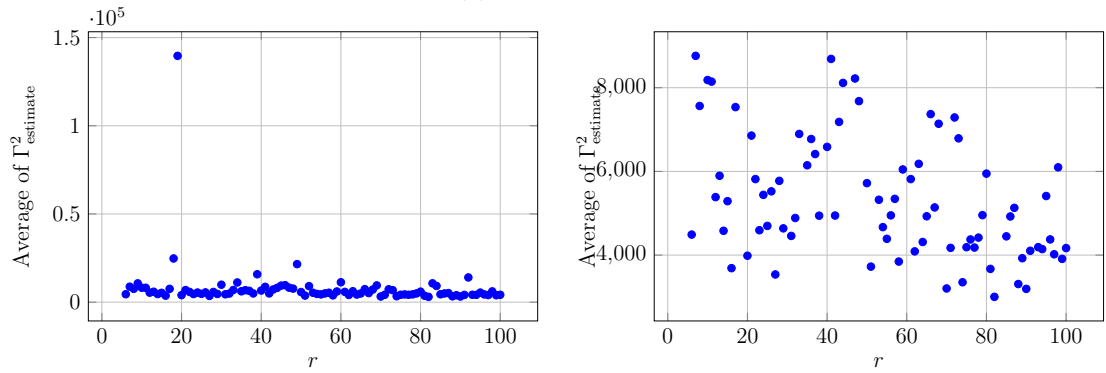

Recall that \Cref{alg:gamma-estimate} returns an estimate $\GammaEstimate$ satisfying \eqref{eq:theoretical-guarantee-example}
with probability at least $1-\epsilon$.  As the upper bound involves a large multiplicative constant, and the estimate has positive probability of failure, occasional large values are observed in the empirical data. In all experiments, most samples lie in a relatively moderate range except for these small number of extreme values that dominate the averages.

To better visualize the typical behavior, we also display versions of the plots with the largest $15$ values removed. In these trimmed plots, the estimates remain comparatively stable across the range of Waring lengths considered, without showing a pronounced growth with respect to $r$. This is consistent with the trend predicted by \Cref{thm:waring-condition-bound}, where the dependence on the Waring length becomes asymptotically negligible.

\subsection{Visualizing roots}\label{ss:cr-visualizing-cool-systems}

We illustrate how $\gammaHatFrob$ in \Cref{alg:rigid_continuation} varies along a path tracked by a root. $\gammaHatFrob$ is a natural value to monitor while tracking, as it incorporates the $\GammaEstimate$ values for each polynomial in the system and is a main component when computing the step size. In \Cref{fig:zero-path}, we run rigid homotopy continuation on systems with configurations $(n,D,r)=(2,3,4)$ and $(n,D,r)=(2,4,5)$. Each curve is colored according to the value of $\gammaHatFrob$; lighter coloring indicates a larger $\gammaHatFrob$. These curves are also projected onto the base for clarity.

Note that larger $\gammaHatFrob$ values often occur in places where at least one path has higher curvature. This matches our expectation, since we must generally take smaller steps to resolve regions with rapid change. Other factors may also contribute to large $\gammaHatFrob$ values (for instance, proximity to singularities). We provide these visualizations as a qualitative view of how $\gammaHatFrob$ evolves along solution paths.

\begin{figure}[hbtp!]
    \centering
    \begin{subfigure}[b]{\textwidth}
        \centering
        \includegraphics[width=0.85\linewidth]{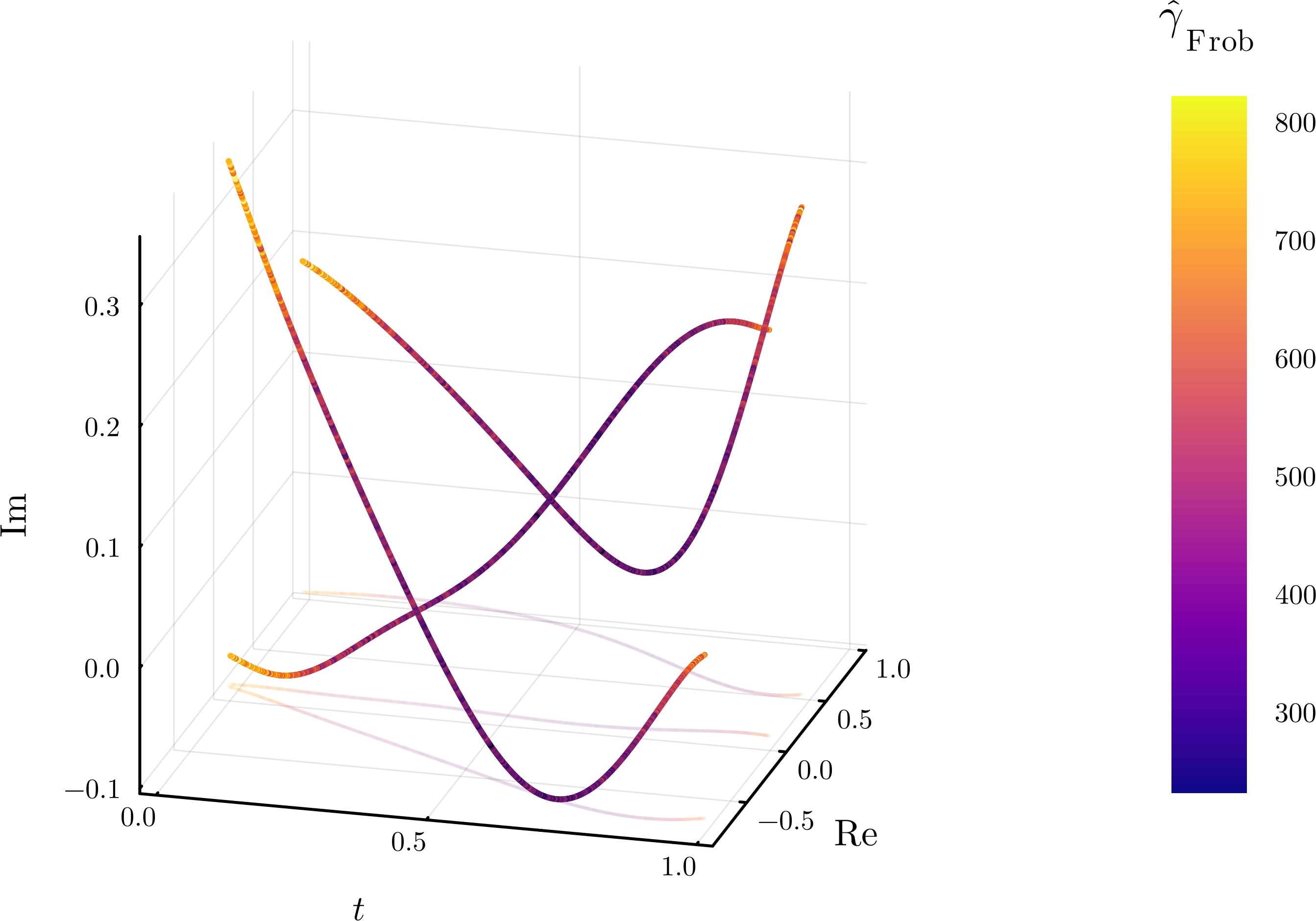}
        \caption{$n=2, D=3, r=4$}
    \end{subfigure}
    \begin{subfigure}[b]{\textwidth}
        \centering
        \includegraphics[width=0.85\linewidth]{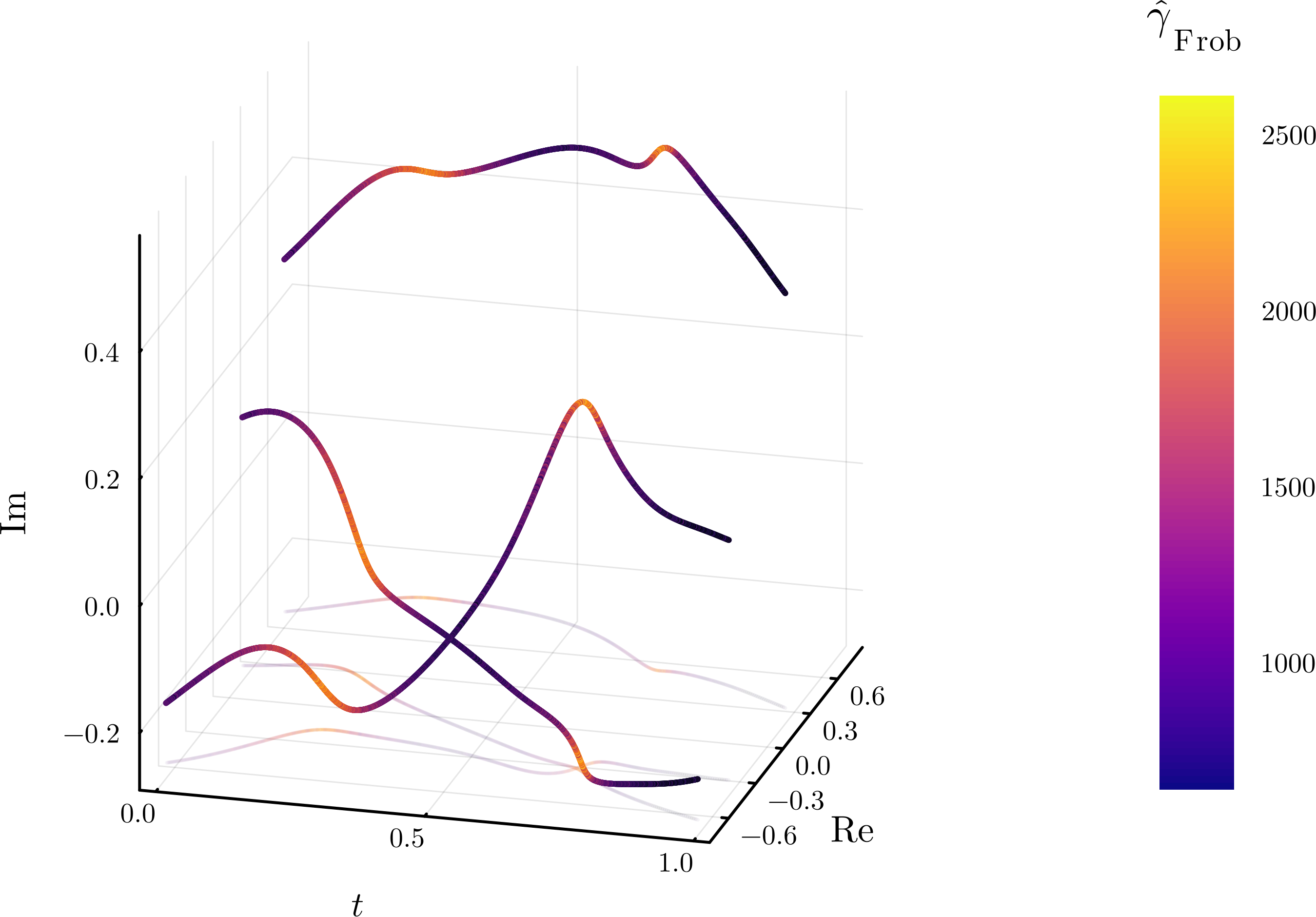}
        \caption{$n=2, D=4, r=5$}
    \end{subfigure}
    \caption{An illustration of a root tracked along a rigid homotopy path for
    two random Waring systems with the given configurations. Each curve
    represents a single coordinate of the root. Top: required
    approximately $1.66 \times 10^5$ iterations to complete. Bottom: required approximately $7.67 \times 10^5$ iterations to complete.}
    \label{fig:zero-path}
\end{figure}

\subsection{Comparison with
heuristic step size selection}\label{ss:cr-heuristics}

We compare \Cref{alg:rigid_continuation} with path tracking using heuristic step sizes. To evaluate the performance of the heuristic approach, we generate $100$ Waring polynomials with configuration $(n,D,r) = (1,3,4)$. For each instance, we sample a starting root and construct a rigid homotopy path. We then track this fixed path using constant heuristic step sizes of the form $10^{-j}$ for $j \in \{1,2,3,4,5\}$.

For each run, we compare the root returned by \Cref{alg:rigid_continuation} with that obtained using the heuristic step size. If the magnitude of the difference between the two is less than $10^{-14}$, we consider the heuristic method to be successful. We report the success rate for each heuristic step size in \Cref{tab:heuristic-vs-certified}.

\begin{table}[ht]
    \centering
    \begin{tabular}{l  c c c c c}
        \toprule
        $j$ & $1$ & $2$ & $3$ & $4$ & $5$  \\
        \midrule
        \begin{tabular}[c]{@{}l@{}}success  \\ rate ($\%$)\end{tabular} & 98 & 100 & 100 & 100 & 100  \\
        \bottomrule
    \end{tabular}
    \caption{Rigid homotopy tracking with a constant heuristic step size of $10^{-j}$. For each $j$, we report the success rate that it converges to the same point obtained by \Cref{alg:rigid_continuation}.}
    \label{tab:heuristic-vs-certified}
    % data in examples/data/complexity/data_final_heuristic_4_*.txt
\end{table}

For this system, the average step size from \Cref{alg:rigid_continuation} is approximately $4.17\times 10^{-5}$. 
In contrast, the constant step sizes $10^{-2}$, $10^{-3}$, and $10^{-4}$ achieved $100\%$ success while being orders of magnitude larger than this average. The smaller step size $10^{-5}$ also achieved $100\%$ success, as expected.

Overall, the results suggest that constant step sizes are often successful in this setting. Moreover, heuristic methods offer advantages in efficiency: larger step sizes reduce the number of continuation steps, and bypassing the computation of $\Gamma_{\mathrm{estimate}}$ further decreases the per-step cost. This motivates the development of robust rigid homotopy tracking methods that allow more flexible step size choices without sacrificing reliability.

\subsection{Remarks on scalability and practical considerations}\label{ss:scalability-considerations}

Most of our experiments are conducted for relatively small $n$ and $D$. 
As $n$ and $D$ increase, we observe that \Cref{alg:rigid_continuation} becomes slower and, in some cases, does not terminate within a reasonable time.  
A primary contributor is the conservative nature of the step size construction in \Cref{alg:gamma-estimate}, which can force excessively small steps (and thus a large number of continuation iterations) that then increase the overall computational burden. 
These observations emphasize the need for more robust and adaptive step size selection strategies in practical implementations, as suggested in \Cref{ss:cr-heuristics}.
    
    \newcommand{\real}{{\operatorname{real}}}
    \newcommand{\imag}{{\operatorname{imag}}}
    \newcommand{\jj}{\texttt{j}}
\newcommand{\rank}{\operatorname{rank}}
\newcommand{\Span}{\operatorname{span}}
\newcommand{\QQ}{\mathbb{Q}}
\newcommand{\ZZ}{\mathbb{Z}}

\section{Conclusion and outlook}\label{s:conclusion-outlook}

In this article we presented theoretical analysis and computational results for systems with Waring representations. Our main result is on the $\gamma$-number for these systems of equations through an analysis of
\eqref{eq:main-expectation}
for rigid homotopy continuation methods. 
In addition, we provided empirical results to complement the theory. 
To conclude, we present an outlook for further research directions. 

\begin{description}
    \item[Certifiable and efficient tracking.] 
    We considered the group $\UBigGroup = U(n+1)^n$, acting on each polynomial $f_i$ for $i=1,\dots,n$. More generally, rigid homotopy continuation applies to a polynomial system $f_i: \CC^{N+1} \to \CC^{n_i}$, with associated varieties $V(f_i)\subset \PP^N$ satisfying $\sum \codim{V(f_i)} = N$. 
    A natural direction for future work is to estimate $\gamma(f_i,\zeta)$ and determine $\GammaFrob(f_i)$ for classes of varieties $V(f_i)$ arising in applied algebraic geometry. However, estimating $\gammaFrob$ using \Cref{alg:gamma-estimate} may significantly overestimate  the value, as illustrated in \Cref{ex:estimateing-for-conic}. This suggests that more refined, variety-specific approaches to estimating $\gammaFrob$ would be beneficial. Such sharper estimates could in turn lead to more efficient and certified rigid homotopy tracking.

\item[Interval arithmetic-based implementation.] 
Even more robust rigid homotopy tracking can be achieved using interval arithmetic. Given recent advances \cite{DL2024-Krawczyk,guillemot2025certified,GL2024-rust-validated-higher-order}, interval-based methods enable certified evaluation and root tracking without relying on idealized assumptions such as the BSS model \cite{blum1989theory}. This provides a pathway toward practical implementations of rigid homotopy that preserve the correctness guarantees of the theoretical framework.

\item[Other efficient evaluation models: polynomial neural networks.] 
\newcommand{\btheta}{\mathbf{\theta}}
\newcommand{\bw}{\mathbf{w}}

One common efficient evaluation model in machine learning is a feed-forward neural network. In general, an $L$-layer \demph{feedforward neural network} $F_\btheta:\mathbb{R}^{d_0}\to\mathbb{R}^{d_L}$ is a composition of affine-linear maps $A_i:\mathbb{R}^{d_{i-1}}\to\mathbb{R}^{d_i}$ and non-linear maps $\sigma_i:\mathbb{R}^{d_i}\to\mathbb{R}^{d_i}$, 
\begin{equation*}
    F_\btheta(x)=(A_L\circ\sigma_{L-1}\circ A_{L-1}\circ\cdots\circ A_2\circ\sigma_1\circ A_1)(x).
\end{equation*}
The \demph{architecture} of the neural network $F_\btheta$ is the sequence $\mathbf{d} = (d_0,d_1,\dotsc,d_L)$. 
The \demph{activation map} $\sigma_i:\mathbb{R}^{d_i}\to\mathbb{R}^{d_i}$ is given coordinate-wise by the \demph{activation function}. When the activation function is a monomial, the neural network is known as a polynomial neural network (PNN). 
Our results can be viewed as completing the analysis of rigid homotopy methods when each polynomial in the system is a  PNN with architecture $(n+1,r,1)$.
Generalizing our complexity results to other
evaluation models used in machine learning is another important direction for future research.

\item[Monodromy and discriminant loci.] 
When homotopy continuation methods are used to compute all solutions of a system, one typically constructs a homotopy by choosing an appropriate start system, often based on the expected number of solutions. In contrast,
rigid homotopy does not come with a natural choice of start system. This absence makes it more challenging to apply rigid homotopy to tasks that require computing all solutions of a system. 
In this context, monodromy methods \cite{DHJLLS2019-monodromy} provide a useful heuristic for exploring the solution set. Also, monodromy techniques offer insight into the geometric structure of the problem, including symmetries and the permutations between solutions \cite{duff2023using,duff2022galois,duff2026certifying}. It will therefore be interesting to study the monodromy actions arising from the unitary transformations underlying rigid homotopy.

One possible approach is to characterize the discriminant loci of rigid homotopies and to build methods using this knowledge. 
To that end, a next step in this direction is to develop algorithms to compute the discriminant locus of rigid homotopies effectively. 
\end{description}

\subsubsection*{Acknowledgments}
We are very thankful for the comments from Peter B{\"u}rgisser 
and 
Alperen Erg{\"u}r at the early stages of this project.
{This research was partially supported by the Alfred P. Sloan Foundation and by the National Science Foundation Grant No. 2510307.}

\bibliographystyle{siam-no-dash-title-color-links}  \bibliography{refs}

\bigskip \medskip \bigskip

\noindent
\footnotesize {\bf Authors' addresses:}
\smallskip

\noindent Abigail R. Jones, University of Wisconsin--Madison, USA \hfill {\tt  abigailjones@math.wisc.edu} \url{https://abigailrjones.github.io/}

\smallskip
\noindent Kisun Lee, 
Clemson University, USA
\hfill {\tt  kisunl@clemson.edu}\newline
\url{https://klee669.github.io}

\smallskip
\noindent Jose Israel Rodriguez, University of Wisconsin--Madison, USA \hfill {\tt  jose@math.wisc.edu}\newline
\url{https://sites.google.com/wisc.edu/jose/}
\end{document}